\documentclass[12pt]{amsart}
\usepackage[left=2.5cm,right=2.5cm, top=2.5cm, bottom=2.5cm]{geometry}
\usepackage{amsmath,amsfonts,amssymb}
\usepackage{color} 
\usepackage{enumerate,graphics, graphicx,tikz}
\usepackage{pgfplots} 
\usepackage[pdftex,bookmarks=true]{hyperref}
\usepackage{verbatim}
\usepackage{caption}
\usepackage{subcaption}
\usepackage[normalem]{ulem}

\usepackage[ruled,vlined]{algorithm2e}

\newtheorem{theorem}{Theorem}[section]
\newtheorem{lemma}[theorem]{Lemma}
\newtheorem{conjecture}[theorem]{Conjecture}
\newtheorem{proposition}[theorem]{Proposition}
\newtheorem{prop}[theorem]{Proposition}
\newtheorem{corollary}[theorem]{Corollary}

\theoremstyle{definition}

\newtheorem{definition}[theorem]{Definition}

\newtheorem{example}[theorem]{Example}
\newtheorem{question}[theorem]{Question}

\newtheorem{proc}[theorem]{Procedure}
\newtheorem{remark}[theorem]{Remark}

\newcommand{\Q}{{\mathbb Q}}
\newcommand{\Z}{{\mathbb Z}}

\newcommand{\C}{{\mathbb C}}

\newcommand{\R}{\mathbb R}

\newcommand{\St}{{S}}
\newcommand{\invtPoly}{\mathcal{P}}

\newcommand{\Rnn}{\mathbb{R}_{\geq 0}}

\def\lra{\leftrightarrows}
\def\rla{\rightleftarrows}
\def\been{\begin{enumerate}}
\def\enen{\end{enumerate}}

\def\SS{\mathcal S}
\def\CC{\mathcal C}
\def\RR{\mathcal R}

\def\lra{\leftrightarrows}
\def\rla{\rightleftarrows}

\title[Absolute concentration robustness]{Absolute concentration robustness: \\ algebra and geometry\footnote{This article is part of the volume titled 
     ``Computational Algebra and Geometry:
     A special issue in memory and honor of Agnes Szanto''.  
     We dedicate this article to the memory of Agnes Szanto. She was a friend and mentor to us, and her kindness and endless support will forever guide us as we endeavor to make a more inclusive algebraic geometry community. 
}}

\author[Garc\'{\i}a Puente]{Luis David Garc\'{\i}a Puente}
\address{Colorado College}
\author[Gross]{Elizabeth Gross}
\address{University of Hawai`i at M\={a}noa}
\author[Harrington]{Heather A Harrington}
\address{University of Oxford and 
Max Planck Institute of Molecular Cell Biology and Genetics and Technische Universit\"at Dresden}
\author[Johnston]{Matthew Johnston}
\address{Lawrence Technological University}
\author[Meshkat]{Nicolette Meshkat}
\address{Santa Clara University}
\author[P\'{e}rez Mill\'an]{Mercedes P\'{e}rez Mill\'{a}n}
\address{Universidad de Buenos Aires and CONICET}
\author[Shiu]{Anne Shiu}
\address{Texas A\&M University}
\date{12 November 2024}

\begin{document}
 \maketitle


\begin{abstract}
Motivated by the question of how biological systems maintain homeostasis in changing environments, Shinar and Feinberg introduced in 2010 the concept of absolute concentration robustness (ACR).
A biochemical system exhibits ACR in some species if the steady-state value of that species does not depend on initial conditions.  Thus, a system with ACR can maintain a constant level of one species even as the initial condition changes.
Despite a great deal of interest in ACR in recent years, the following basic question remains open: How can we determine quickly whether a given biochemical system has ACR?  Although various approaches to this problem have been proposed, we show that they are incomplete. Accordingly, we present new methods for deciding ACR, which harness computational algebra.  We
illustrate our results on several biochemical signaling networks.

\vspace{.1in}
\end{abstract}

\maketitle
\section{Introduction} \label{sec:intro}

How do cells maintain function in fluctuating environments? 
In signaling transduction pathways, such fluctuations might be in the abundances of signaling proteins or may arise from 
crosstalk between proteins.  
Nevertheless, robustness of an output signal has been observed experimentally.  In \emph{E.~coli}, for instance, such robustness is found in the EnvZ-OmpR osmoregularity signaling pathway \cite{Batchelor2003,Shinar2007} and the glyoxylate bypass mechanism \cite{LaPorte,Shinar2009}, often over several orders of magnitude in the abundances of individual signaling proteins.

Mathematically, this phenomenon was first rigorously studied by Shinar and Feinberg who introduced the concept of {\em absolute concentration robustness} (ACR)~\cite{shinar2010structural}.  A biochemical system exhibits ACR in some species $X_i$ if at every positive steady state, regardless of initial conditions, the concentration of $X_i$ is the same.  
Thus, a system with ACR maintains species $X_i$ at a constant level, even under changes to the initial condition.

Most prior results regarding ACR are sufficient conditions for ACR.  
The first such criterion, due to Shinar and Feinberg, can be verified easily (as implemented in~\cite{kuwahara2017acre})
from the structure of the network of biochemical interactions \cite{shinar2010structural}. Notably, this sufficient condition for ACR is also necessary for some classes of networks~\cite{meshkat-shiu-torres}.  
Other sufficient criteria for ACR have harnessed methods based on elementary modes~\cite{neigenfind2011biochemical,neigenfind2013relation} and
the theory of network translation~\cite{Johnston2014,tonello2017network}. 
More recently, the concept of ``local ACR" has been defined and analysed as a necessary condition for ACR~\cite{beatriz-elisenda}.

Other studies have focused on developing and applying results on ACR to understand the biochemical mechanisms underlying robustness, for example, in mechanisms with dimerization and bifunctional proteins \cite{dexter2015,dexter2013dimerization}. 
This topic of biological robustness has generated interest in ACR from control theory~\cite{hidden, controller} and has also inspired a refined definition of ACR which better captures the types of robustness relevant to applications~\cite{joshi-craciun,joshi-craciun-2}. 
Additionally, the property of ACR has been studied in randomly generated networks~\cite{joshi2023prevalence} and in models that are stochastic, rather than deterministic~\cite{Anderson2016, A-E-J,Enciso2016}.

ACR has also been analyzed using methods from computational algebra~\cite{kaihnsa2023absolute, karp2012complex, beatriz-elisenda}. 
Our work builds on those works to relate ACR to various algebraic objects (ideals and varieties) associated to a (biochemical or mass-action) system. 
We also distinguish between ACR of a given system and ACR for a family of systems arising from a given reaction network. More generally, we explore the underlying algebraic and geometric structure of ACR.

We also consider the problem of deciding whether a network or system has ACR. 
We first 
show by counterexample that various approaches (some in the literature) to this decision problem are incomplete (\S~\ref{sec:how-not-to-decide-ACR})
Accordingly, 
we develop new procedures, using computer algebra, to tackle this problem for mass-action systems (nevertheless, many of our examples illustrate these procedures without any dependence on the rate constants and so allow us to draw conclusions about ACR at the network level).  
Specifically, we use a new ideal, which we call the positive-restriction ideal (Section~\ref{sec:comp_considerations}), specializations of Gr\"obner bases (Section~\ref{sec:specialized-GB}) and
numerical algebraic geometry (Section~\ref{sec:numerical}). 

We are not the first to harness numerical algebraic geometry~\cite{case-study, bates-gunawardena} and specializations of Gr\"obner bases~\cite{DPST} to prove results about biochemical systems.  
Indeed, a great deal of research has been conducted recently using algebraic techniques to analyze steady states of mass-action systems.  
Many such investigations concern when steady states are defined by binomial equations~ \cite{TDS, Johnston2014, TSS}.  
Other studies have focused on computing steady-state invariants \cite{karp2012complex}, constructing and harnessing steady-state parameterizations \cite{CFMW, DPST, GRPD, J-M-P, OSTT, messi}, 
computing the maximum number of steady states~\cite{case-study, mv-small-networks}, and characterizing bistability and oscillations~\cite{mixed,erk-limit, torres-feliu}.

The outline of our work is as follows. 
We introduce mass-action systems,  and then examine the relationship between ACR and multistationarity, steady-state parameterizations, and 1-species networks
in Section~\ref{sec:bkrd-CRS}.  
In Section~\ref{sec:comp_considerations}, we consider computational aspects of assessing ACR, we present several examples which illustrate the limitations to sufficiency and necessity of the presented conditions with respect to establishing ACR, and we introduce an ideal-decomposition algorithm for establishing ACR.
Section~\ref{sec:specialized-GB} investigates methods for detecting ``zero-divisor ACR.''
In Section~\ref{sec:numerical}, we present numerical methods to detect or to preclude ACR, and we end with a discussion in Section~\ref{sec:disc}.

\section{Mass-action systems and ACR} \label{sec:bkrd-CRS}
This section introduces chemical reaction networks (Section~\ref{sec:CRN}), mass-action systems (Section~\ref{sec:CRS}), their steady states (Section~\ref{subsec:steady-states}), 
and ACR (Section~\ref{sec:ACR}). We also assess the problem of deciding ACR for several classes of networks
(Sections~\ref{sec:mss-acr}--\ref{sec:1-species-algo}).

\subsection{Chemical reaction networks} \label{sec:CRN}

\begin{definition} \label{def:crn}
A {\em chemical reaction network} $G=(\SS,\CC,\RR)$ consists of three finite sets:
\begin{enumerate}
\item a set of chemical {\em species} $\SS = \{A_1,A_2,\dots, A_n \}$, 
\item a set  $\CC = \{y_1, y_2, \dots, y_p\}$ of {\em complexes} (finite nonnegative-integer combinations of the species), and 
\item a set of {\em reactions}, ordered pairs of the complexes: $\RR \subseteq (\CC \times \CC) \smallsetminus \{ (y,y) \mid y \in \CC\}$.
\end{enumerate}
\end{definition}

A network can be viewed as a directed graph whose nodes are complexes and whose edges correspond to the reactions (and so an edge $(y,y')$ can be denoted by $y \to y'$). 

Next, writing the $i$-th complex as $y_{i1} A_1 + y_{i2} A_2 + \cdots + y_{in}A_n$ (where $y_{ij} \in \mathbb{Z}_{\geq 0}$ for $j=1,2,\dots,n$), 
we introduce the following monomial:
$$ x^{y_i} \,\,\, := \,\,\, x_1^{y_{i1}} x_2^{y_{i2}} \cdots  x_n^{y_{in}}~. $$
(By convention, the {\em zero complex} yields the monomial $x^{(0,\dots,0)}=1$.)  
The vectors $y_i$ define the rows of a $p \times n$-matrix of nonnegative integers, which we denote by $Y=(y_{ij})$.
Next, the unknowns $x_1,x_2,\ldots,x_n$ represent the concentrations of the $n$ species in the network, and we regard them as functions $x_i(t)$ of time $t$.

For a reaction $y_i \to y_j$ from the $i$-th complex to the $j$-th complex, the {\em reaction vector}  $y_j-y_i$ encodes the net change in each species that results when the reaction takes place.  
The {\em stoichiometric subspace} is the vector subspace of $\mathbb{R}^n$ spanned by the reaction vectors $y_j-y_i$, and we will denote this space by $\St$: 
\begin{equation} \label{eq:stoic_subs}
  \St~:=~ \mathbb{R} \{ y_j-y_i \mid  y_i \to y_j~ {\rm is~in~} \RR \}.
\end{equation}

Networks appearing in our examples have only a few species, and so we denote the species by $A, B, C, \dots$ instead of $X_1, X_2, X_3 \dots $.
  
 \begin{example} \label{ex:1-rxn}
Consider the network $G = \{A+B \to 3A + C\}$, which consists of a single reaction (with 3 species and 2 complexes). We have $y_2-y_1 =(2,-1,1)$, which
means that with each occurrence of the reaction, two units of $A$ and one of $C$ are produced, while one unit of $B$ is consumed.  This vector $(2,-1,1)$ spans the
stoichiometric subspace $\St$. 
\end{example}

\subsection{Mass-action systems} \label{sec:CRS}

The dynamical systems considered in this work come from assigning mass-action kinetics to a chemical reaction network.  
According to the law of mass-action, the rate of each chemical reaction is directly proportional to the product of the concentrations of the reactants.
Therefore, we associate to each reaction $y_i \to y_j$ of a network, 
a positive parameter $\kappa_{ij}$, the {\em rate constant} of the reaction. 
By ordering the $r$ reactions, a choice of rate constants can be represented as a vector
$(\kappa_{ij}) \in
\mathbb{R}^{r}_{>0}$.

\begin{definition} \label{def:CRS} 
A {\em mass-action system}  $(G,\kappa)$ refers to the  dynamical system arising from a specific chemical reaction network $G=(\SS, \CC, \RR)$ and a choice of rate parameters $\kappa=(\kappa_{ij}) \in
\mathbb{R}^{r}_{>0}$ (here $r$ denotes the number of
reactions), as follows: 
\begin{align} \label{eq:ODE-mass-action}
\frac{dx}{dt} \quad = \quad \sum_{ y_i \to y_j~ {\rm is~in~} \RR} \kappa_{ij} x^{y_i}(y_j - y_i) \quad =: \quad f_{\kappa}(x)~=~ (f_{\kappa}(x)_1, f_{\kappa}(x)_2, \dots, f_{\kappa}(x)_n)~.
\end{align}
\end{definition}

In this article, we treat the rate constants $\kappa_{ij}$ as (positive) unknowns in order to analyze the entire family of dynamical systems that arise from a given network, as the $\kappa_{ij}$'s vary.

By construction, the  vector $\frac{d x}{dt}$ in  (\ref{eq:ODE-mass-action}) lies in $\St$ for all time $t$.   
Additionally, the positive orthant is forward-invariant~\cite{Vol72}.  Hence, a trajectory $x(t)$ beginning at a nonnegative 
vector $x(0)=x^0 \in \R^n_{\geq 0}$ remains in the {\em stoichiometric compatibility class}, which we denote by
\begin{align}\label{eqn:invtPoly}
\invtPoly~:=~(x^0+\St) \cap \mathbb{R}^n_{\geq 0}~, 
\end{align}
for all positive time.  In other words, $\invtPoly$ is forward-invariant with respect to the dynamics~(\ref{eq:ODE-mass-action}).    

\begin{example} \label{ex:generalized-shinar--Feinberg}
The following is a ``generalized Shinar--Feinberg network''~\cite{meshkat-shiu-torres}:
\begin{align} \label{eq:gen-SF-net}
\{ B \to A, ~ 2A+B \to A+2B \}~.    
\end{align}
The mass-action ODEs~\eqref{eq:ODE-mass-action} arising from this network are as follows:
 \begin{align} \label{eq:ode-shinar--Feinberg}
    \frac{dx_A}{dt} ~&=~ \kappa_1 x_B - \kappa_2 x_A^2 x_B \\
    \frac{dx_B}{dt} ~&=~ -\kappa_1 x_B + \kappa_2 x_A^2 x_B~, \notag
 \end{align}
 where $\kappa_1$ and $\kappa_2$, respectively, are the rate constants for the first and second reactions in~\eqref{eq:gen-SF-net}.  
 The stoichiometric subspace is spanned by the vector $(1,-1)$, so the stoichiometric compatibility classes are the line segments defined by $c \in \mathbb{R}_{>0}$, as follows:
 \begin{align} \label{eq:example-SCC}
 \invtPoly_c ~=~ \{ (x_A,~x_B) \in \mathbb{R}^{2}_{\geq 0} \mid x_A+x_B=c\}~.
 \end{align} 
\end{example}

\subsection{Steady states and related ideals} \label{subsec:steady-states}

Here we define steady states and the polynomial ideals that they generate. 
\begin{definition} \label{def:steady-ste} 
Consider a mass-action system $(G, \kappa)$ with $n$ species. 
\begin{enumerate}
	\item A {\em steady state} 
	is a nonnegative concentration vector $x^* \in \Rnn^n$ at which the ODEs~\eqref{eq:ODE-mass-action}  vanish: $f_{\kappa} (x^*) = 0$.  
We distinguish between {\em positive steady states} $x ^* \in \mathbb{R}^n_{> 0}$ and {\em boundary steady
states} $x^* \in  \left( \mathbb{R}^n_{\geq 0} \smallsetminus \mathbb{R}^n_{> 0} \right)$.  
	\item The {\em positive steady-state locus} 
	is the set of all positive steady states.
	\item The {\em steady-state ideal} is the ideal in the polynomial ring $\mathbb R [x_1, x_2, \ldots , x_n] $ 
	generated by the right-hand sides of the ODEs in~\eqref{eq:ODE-mass-action}:
	\[
	I(G,\kappa) ~:=~
	\langle
	f_{\kappa}(x)_1, ~f_{\kappa}(x)_2, ~\dots, ~f_{\kappa}(x)_n
	\rangle~.
	\]
	\item The \emph{ideal of the positive steady-state locus} is defined as follows:
$$
\sqrt[>0]{I(G,\kappa) }~:=~
 \left\{ h \in \mathbb R[x_1, x_2, \ldots , x_n] \ | \ h(x) = 0 \text{ for all positive steady states } x \right\} ~.$$
\end{enumerate}
\end{definition}

We recall the definition of the \emph{radical} of an ideal 
$I$ of a ring $R$:
\begin{equation}\label{eq:radical}
 \sqrt{I}~:=~ \left\{ h \in R \ | \ h^m \in I\text{ for some } m\in\Z_{>0} \right\}~.
\end{equation}

\begin{remark}  \label{rmk:levels-invariants}
By construction, we have $I(G,\kappa) 
\subseteq \sqrt[>0]{I(G,\kappa) }$.  This containment of ideals can be refined, yielding three ``levels'' of ``steady-state invariants''
\cite[equation~(7)]{Dickenstein:Invitation}:
\begin{align} \label{eq:ideal-containments}
I(G,\kappa) 
~\subseteq~
\sqrt{I(G,\kappa)}
~\subseteq~
\sqrt[{> 0}]{I(G,\kappa)}~.
\end{align}
The fact that the containments in~\eqref{eq:ideal-containments} are, in general, strict, was noted by Dickenstein~\cite{Dickenstein:Invitation}\footnote{We make a small clarification: Dickenstein considered polynomials vanishing on {\em all} steady states, whereas we consider only positive steady states.  
As a result, the ideal Dickenstein called the ``positive real radical of the steady-state ideal'' is closely related to, but not the same as, our 
ideal $\sqrt[{> 0}]{I(G,\kappa)}$.} 
and is the reason behind many of the ``warning'' examples shown later in Section~\ref{sec:how-not-to-decide-ACR}.  
\end{remark}

\begin{example} \label{ex:steady-state-ideals}
Consider the network $G=\{ 2A \to A,~ 3A \to 5A,~ 4A \to 3A\}$.  When all three rate constants are taken to be 1 (that is, $\kappa^*=(1,1,1)$), 
the resulting mass-action ODE is the following:
\[
	\frac{dx_A}{dt} ~=~ - x_A^4 + 2x_A^3 -x_A ^2 ~=~ - x_A^2(x_A-1)^2~. 
\]
(Notably, $x_A^*=1$ is the unique positive steady state.) 
The ideals in~\eqref{eq:ideal-containments} are thus as follows:
\[
I(G,\kappa^*) 
~=~
\langle x_A^2 (x_A-1)^2 \rangle
~\subsetneq~
\sqrt{I(G,\kappa^*)}
~=~
\langle x_A(x_A-1) \rangle
~\subsetneq~
\sqrt[{> 0}]{I(G,\kappa^*)}
~=~
\langle x_A-1 \rangle~.
\]
\end{example}

Next, we move from systems $(G,\kappa)$ to networks $G$ by viewing $\kappa$ as unknown.

\begin{definition} \label{def:steady-ste-locus} 
Let $G$ be a chemical reaction network with $n$ species and $r$ reactions. 
\begin{enumerate}
 	\item $G$ is {\em multistationary} if there exist positive rate constants $\kappa \in \mathbb{R}^r_{>0}$ such that the resulting mass-action system~\eqref{eq:ODE-mass-action} admits
two or more positive steady states in some stoichiometric compatibility class~\eqref{eqn:invtPoly}. 
	\item The {\em positive steady-state/rate locus} is the set of all pairs of positive steady states and corresponding rate constants:
	\[
	\mathcal{L} ~:=~
	\left\{ (x,\kappa) \in \mathbb{R}^n_{>0} \times \mathbb{R}^r_{>0} \mid
		x  \text{ is a steady state of the mass-action system } (G, \kappa)  \right\} ~.
	\]
	\item The {\em ideal of the positive steady-state/rate locus} is:
	\[
	I(\mathcal{L}) ~:=~
		 \left\{ h \in \mathbb R[x_1, x_2, \ldots , x_n; \kappa_{ij}] \mid 
		 h(x,\kappa)=0  \text{ for all } (x,\kappa) \in \mathcal{L} \right\}~. 
	\]
\end{enumerate}
\end{definition}

\begin{example}[Example~\ref{ex:generalized-shinar--Feinberg}, continued] \label{ex:ideals-network}
For the network $G=\{ B \overset{\kappa_1}\to A, ~ 2A+B \overset{\kappa_2}\to A+2B \}$, 
it is straightforward to compute the positive steady-state/rate locus:
\begin{align} \label{eq:locus-for-generalized-sf-net}
	\mathcal{L} ~=~
		\left\{ (x,\kappa) \in \mathbb{R}^2_{>0} \times \mathbb{R}^2_{>0} \mid x_A = \sqrt{\kappa_1/ \kappa_2} \right\}~.
\end{align}
It follows that the ideal of the positive steady-state/rate locus is the following:
	\begin{align} \label{eq:ideal-for-SF-net}
	I(\mathcal{L}) ~=~
		\langle
		\kappa_1 - \kappa_2 x_A^2
		\rangle~.
		\end{align}
We observe that, for every choice of positive rate constants $\kappa^*= (\kappa^*_1, \kappa^*_2) \in \mathbb{R}^2_{>0}$, every positive steady state $(x_A^*, x_B^*)$ of the mass-action system $(G,\kappa^*)$ has the same value of $x_A^*$, namely, $x^*_A = \sqrt{\kappa^*_1/ \kappa^*_2}$.  In the following subsection, we give a name to this property.
\end{example}

\begin{remark} \label{rmk:hungarian-lemma}
Any positive part of a variety is equal to the positive steady-state locus of some mass-action system, 
because we can just multiply all the equations by $x_1x_2 \dots x_n$ so that the resulting equations are mass-action~\cite[\S 2]{Dickenstein:Invitation}.  Indeed, many of the examples in this work were generated through such a trick.
\end{remark}

\subsection{Absolute concentration robustness} \label{sec:ACR}
In this subsection, we recall the definition of ACR for mass-action systems and networks (Definitions~\ref{def:ACR} and~\ref{def:ACR-conditional}).
We also relate ACR to steady-state ideals (Propositions~\ref{prop:ideal} and~\ref{prop:ideal-symbol}).

\subsubsection{ACR for mass-action systems}

The following definition 
was introduced by Shinar and Feinberg~\cite{shinar2010structural}. 

\begin{definition}[ACR for mass-action systems] \label{def:ACR}
A mass-action system $( G, \kappa)$ has \emph{absolute concentration robustness (ACR) in species $X_i$} if, for every positive steady state $x$, the value of $x_i$ is the same.  
This value of $x_i$ is the {\em ACR-value}.
\end{definition} 

For simplicity, we say that $(G,\kappa)$ has ACR if it has ACR is some species. 

\begin{definition}[Vacuous ACR] \label{def:vacuous}
A mass-action system $( G, \kappa)$ has {\em vacuous ACR} if it has \uline{no} positive steady states (or, equivalently, the ideal of the positive steady-state locus is the improper ideal $\langle 1 \rangle$).  
\end{definition} 

An example of a system with vacuous ACR 
arises from the network $G=\{A \to B\}$.
Vacuous ACR does not confer any ``robustness'' to the system, and hence our interest is in non-vacuous ACR.
Also, ACR is most relevant for systems having linear conservation laws (that is, the stoichiometric subspace~\eqref{eq:stoic_subs} is a proper subspace of $\mathbb{R}^n$, where $n$ is the number of species), as this is the situation when there is more than one compatibility class and 
so we can compare steady states in distinct classes.
Nevertheless, for ease of presentation, we include some examples without such linear conservation laws.

\begin{example}[Example~\ref{ex:steady-state-ideals} continued] \label{ex:1-species-ACR}
We saw earlier that the mass-action system given by the network
$G=\{ 2A \to A,~ 3A \to 5A,~ 4A \to 3A\}$ and rate constants $\kappa^*=(1,1,1)$ has a unique positive steady state (namely, $x_A^*=1$).  It follows that this system has ACR (in $A$). 
\end{example}

\begin{example}[Example~\ref{ex:ideals-network} continued] \label{ex:generalized-shinar--Feinberg-2}
We observed above that, for $G=\{ B \to A, ~ 2A+B \to A+2B \}$, the mass-action system 
$(G,\kappa^*)$ arising from any choice of positive rate constants $\kappa$ has ACR in $A$.
\end{example}


The following result, which was observed in~\cite{karp2012complex}, relates ACR with the ideal of the positive steady state locus (from Definition~\ref{def:steady-ste}). 

\begin{prop}[ACR and ideals] \label{prop:ideal}
A mass-action system $( G, \kappa)$ has ACR in species $X_i$ if and only if $x_i - \alpha$, for some $\alpha >0$, is in the ideal of the positive steady state locus of $(G,\kappa)$. 
\end{prop}

\begin{proof}
If $(G,\kappa)$ has no positive steady states (that is, $\sqrt[>0]{I(G,\kappa) } =  \langle 1 \rangle$), then $x_i-1$ is in $\sqrt[>0]{I(G,\kappa)}$ and $(G,\kappa)$ vacuously has ACR (recall Definition~\ref{def:vacuous}).  Now assume that $(G,\kappa)$ has at least one positive steady state. 
Then $( G, \kappa)$ has ACR in $X_i$ with ACR-value $\alpha>0$ if and only if $x^*_i=\alpha$ for every positive steady state $x^*$.  The result now follows directly.  \end{proof}

\begin{remark} \label{rem:local-acar}
    A necessary condition for ACR, called ``local ACR", for mass-action systems (i.e., with fixed rate constants) appears in~\cite[Theorem 4.11]{beatriz-elisenda} and is an easy-to-check linear algebra condition, with a refinement for the case where the rate constants are not fixed in \cite[Theorem 4.15]{elisenda-oskar-beatriz}.
\end{remark}

\subsubsection{ACR for networks}
If 
$(G,\kappa)$ has ACR, then does ACR persist when $\kappa$ is replaced by another choice of rate constants?  This question motivates the following definition.

\begin{definition}[ACR for networks] \label{def:ACR-conditional}
A network $ G$ with $r$ reactions has:
	\begin{enumerate}[(i)]
	\item the {\em capacity for ACR} if there exist a species $X_i$ and some $\kappa \in \mathbb R_{>0}^{r}$ such that $( G, \kappa)$ has {ACR} in $X_i$; 
	\item {\em unconditional ACR} (or, simply, {\em ACR}) if there exists a species $X_i$ such that for all $\kappa \in \mathbb R_{>0}^{r}$, the mass-action system $( G, \kappa)$ has ACR in $X_i$.
	\end{enumerate}
\end{definition}

Notice that networks with (unconditional) ACR automatically have the capacity for ACR.  These ideas are related to ``hybrid robustness'' \cite[pg.\ 889]{dexter2015}.

\begin{example}[Example~\ref{ex:generalized-shinar--Feinberg-2} continued] \label{ex:generalized-shinar--Feinberg-3}
From our earlier observations, we see that the network $\{ B \to A, ~ 2A+B \to A+2B \}$ has (unconditional) ACR (in species $A$, but not $B$).
\end{example}

\begin{example}[Example~\ref{ex:1-species-ACR} continued] \label{ex:capacity}
We saw that the network $G=\{ 2A \to A,~ 3A \to 5A,~ 4A \to 3A\}$ has the capacity for ACR.  
We now claim that $G$ does {\em not} have (unconditional) ACR.  
To see this, consider the rate constants $\kappa^*=(2,1.5,1)$.  
The resulting ODE is 
\[
	\frac{dx_A}{dt} ~=~ - x_A^4 + 3x_A^3 -2x_A ^2 ~=~ - x_A^2(x_A-1)(x_A-2)~,
\]
and so there are two positive steady states, $x_A^*=1$ and $x_A^*=2$.  Hence, $(G,\kappa^*)$ does not have ACR, which verifies the claim.
\end{example}

Recall that Proposition~\ref{prop:ideal} pertains to ACR in mass-action systems (i.e., with specialized rate constants).  We next state a version of that result for networks (i.e., with symbolic rate constants).  Note that while Proposition~\ref{prop:ideal} gives an equivalence between ACR and a condition on a certain ideal, the following result asserts only one such implication.

\begin{prop}[ACR and $ I(\mathcal{L})$] \label{prop:ideal-symbol}
Let $X_i$ be a species of a reaction network $G$.  
Let $I(\mathcal L)$ denote the ideal of the positive steady-state/rate locus of $G$.  
If $x_i - \alpha \in I(\mathcal L)$, for some 
$\alpha \in \mathbb{R}[\kappa_{ij} \mid i \to j \mathrm{~is~a~reaction~of~} G ]$, 
then $G$ has ACR in $X_i$.
\end{prop}

\begin{proof}
Assume 
there exists $\alpha \in \mathbb{R}[\kappa_{ij} \mid i \to j \mathrm{~is~a~reaction~of~} G ]$ 
such that $x_i - \alpha \in I(\mathcal{L})$. 
Fix $\kappa^* \in \mathbb{R}^r_{>0}$, where $r$ is the number of reactions.  

If the mass-action system $(G,\kappa^*)$ has no positive steady states, then the system vacuously has ACR.  
Now assume that $(G,\kappa^*)$ has a positive steady state $x^*$.  
Then, $x_i - \alpha \in I(\mathcal{L})$ implies that $x^*_i = \alpha|_{\kappa  =  \kappa^*}$. 
Thus, $(G,\kappa^*)$ has ACR in $X_i$ with ACR-value $\alpha|_{\kappa  =  \kappa^*}$.  
\end{proof}

The converse of Proposition~\ref{prop:ideal-symbol} is false, as the following example shows.

\begin{example}[Example~\ref{ex:generalized-shinar--Feinberg-3} continued] \label{ex:generalized-shinar--Feinberg-4}
We noted earlier that the network $\{ B \to A, ~ 2A+B \to A+2B \}$ has ACR in $A$, but the ideal $I(\mathcal L)$, in~\eqref{eq:ideal-for-SF-net}, does {\em not} contain a polynomial of the form $x_A - \alpha$ with $\alpha \in \mathbb{R}[\kappa_1,\kappa_2]$.
\end{example}

\subsubsection{ACR and multistationarity} \label{sec:mss-acr}
The relationship between ACR and multistationarity was explored in prior works, in the context of randomly generated reaction networks~\cite{joshi2023prevalence} and for the purpose of determining the minimum numbers of species, complexes, and reactions needed for ACR and multistationarity to coexist~\cite{kaihnsa2023absolute}.
Here, we record two facts.  
First, it is immediate from definitions that if a reaction network $G$ has ACR in \emph{every} species, then $G$ is \emph{not} multistationary. Second, it is also straightforward to check that, for a network $G$ with $n$ species, if the stoichiometric subspace is  $\mathbb{R}^n$, then $G$ has ACR in \emph{every} species $X_i$ if and only if  $G$ is \emph{not} multistationary. 

\subsubsection{Steady-state parameterizations}

For many networks arising in biology, the positive steady-state locus can be parameterized~\cite{perspective, feliu-wiuf-ptm, J-M-P, messi, TG}.  
More precisely, we say that a mass-action system $(G,\kappa)$ with 
$n$ species 
{\em admits a positive steady-state parameterization} if there exists a function:
	\begin{align} \label{eq:param}
	\phi:~	\mathbb{R}^T_{>0} 
				 ~ &\to ~ \mathbb{R}^n_{>0} 
	\end{align}
(for some $T>0$) such that the image 
 equals the set of all positive steady states of $(G,\kappa)$.
When a mass-action system $(G,\kappa)$ admits a positive steady-state parameterization, assessing ACR is easy: $(G,\kappa)$ has ACR in species $X_i$ if and only if 
$\phi_i$ is constant. 

%

\subsubsection{One-species networks and systems} \label{sec:1-species-algo}
For networks having only one species, ACR has already been classified: Such a network has ACR if and only if it is non-multistationary, which in turn is equivalent to a combinatorial criterion on its ``arrow diagram''~\cite[Proposition~4.2]{meshkat-shiu-torres}.  

If a one-species system has fixed rational-number rate constants, ACR means that the (univariate polynomial) right-hand side of the (unique) ODE has at most one positive root.  This property can be checked using the classical theory of Sturm sequences.

\section{Ideal theoretic approaches to detecting ACR}
\label{sec:comp_considerations}

In this section we recall several facts about ideals in polynomial rings, which we use later (Section~\ref{sec:ideals-in-poly-rings}). Section~\ref{sec:how-not-to-decide-ACR}
serves as a warning, by showing that various sufficient conditions for ACR are not necessary.
We then turn our attention to the real radical ideal and discuss the assumption of rational-number rate constants, as well as other computational considerations. We end this section by 
focusing on a new ideal, which we call the positive-restriction ideal (Section~\ref{sec:using-j-ideal}), and by computing decompositions of its radical (Section~\ref{sec:ideal_decomp}) to assess ACR.

\subsection{Ideals in polynomial rings} \label{sec:ideals-in-poly-rings}
We start this section by recalling several concepts related to general ideals (in polynomial rings): saturations, zero-divisors, real varieties, and real radicals.  We use these ideas in later sections (for steady-state ideals and other ideals).  

\begin{definition}\label{def:colon}
 Let 
 $h \in \mathbb{F}[x_1, x_2, \dots, x_n]$, where $\mathbb{F}$ is a characteristic-zero field, 
 and let $J$ be an ideal of  $\mathbb{F}[x_1, x_2, \dots, x_n]$. 
 The \emph{ideal quotient} of $J$ with respect to $h$ is the ideal defined by
 \[(J\colon h) := \{g\in \mathbb{F}[x_1, x_2, \dots, x_n] \mid hg \in J\},\]
 and the \emph{saturation} of $J$ with respect to $h$ is the ideal defined by 
 $$(J:h^\infty) ~:=~ \{ g \in \mathbb{F}[x_1, x_2, \dots, x_n] \mid h^kg \in J \; \text{for some} \; k \in\mathbb{Z}_{\geq 0}  \}.$$
\end{definition}

\begin{remark} \label{rmk:saturation}
For the saturations we consider in this work, the polynomial $h$ is the monomial $\mathfrak{m}$ formed by the product of all the variables $x_i$ (that is, $\mathfrak{m} = x_1 x_2 \cdots x_n$). 
Roughly speaking, $(J:\mathfrak{m}^\infty)$ allows us to ``divide'' the polynomials in $J$ by monomials in the variables. 
Hence, the zeros of $J$ with nonzero coordinates are the zeros of $(J:\mathfrak{m}^\infty)$ with nonzero coordinates. 
Moreover, computing saturation is a standard task accomplished via Gr\"obner bases, which are implemented 
for instance 
in most computer algebra systems.
\end{remark}

Next, we turn to zero-divisors.

\begin{definition} \label{def:zero-divisor}
Let $I$ be an ideal in a ring $R$.
 If $fg \in I$, where $f,g \in R \smallsetminus I$, then $f$ and $g$ are {\em zero-divisors} of $I$.
\end{definition}

\begin{remark} \label{rmk:zero-divisor}
In Definition~\ref{def:zero-divisor}, we use the term ``zero-divisor'' because $f$ is a zero-divisor of $I$ if and only if $f+I$ is a zero-divisor of the quotient ring $R/I$.
\end{remark}

\begin{remark}[Checking zero-divisors] \label{rmk:check-zero-divisor}
A polynomial $f$ is a zero-divisor of an ideal $I$ in a polynomial ring, if and only if the  ideal quotient $(I: f)$ is strictly larger than $I$.  Checking this is easy using computer algebra software (as mentioned in Remark~\ref{rmk:saturation}, Gr\"obner bases can be used to compute saturations -- and the same is true for testing equality of ideals).
\end{remark}

Finally, we consider real varieties and real radicals.  

\begin{definition} \label{def:real-variety}
Let $J$ be an ideal in the polynomial ring $\mathbb{Q}[x_1,x_2,\dots, x_n]$. 
The {\em real radical} of $J$ is the following ideal:
	\begin{align*}
		\sqrt[{\mathbb R}]{J} ~:=~ 
			\left\lbrace g\in \Q[x_1,x_2,\dots, x_n] ~|~  
        \exists m,\ell\in \mathbb{Z}_{\geq 0}, ~\exists h_1,\dots, h_\ell \in \Q[x_1, x_2, \dots, x_n], \;  g^{2m}+\sum_{i=1}^\ell h_i^2 \in J \right\rbrace~.
	\end{align*}
The {\em real variety} of $J$ is $
		V_{\mathbb R} (J) := \left\{ z \in \mathbb{R}^n \mid f(z)=0 \text{ for all } f \in J \right\}$.
\end{definition}

The real radical and real variety are related by the next result, the Real Nullstellensatz (specifically, the version in which the coefficient field is $\mathbb Q$).  
This result is well known; for instance, it follows from~\cite[Theorem 1.11]{lombardi-perrucci-roy}\footnote{A proof of Proposition~\ref{prop:real-nullstellensatz} is in Appendix~\ref{sec:appendix-radical}.}. 
Neuhaus~\cite{Neuhaus1998} credits the result to 
Dubois~\cite{dubois-1969}, Krivine~\cite{krivine-1964}, and Risler~\cite{risler-1970}.

\begin{proposition}[Real Nullstellensatz] \label{prop:real-nullstellensatz}
If $J$ is an ideal in the polynomial ring $\mathbb{Q}[x_1,x_2,\dots, x_n]$, then the 
real radical of $J$ equals the vanishing ideal of the real variety of $J$:
	\begin{align*}
		\sqrt[{\mathbb R}]{J} ~:=~ 
				\left\{ 
				f \in \mathbb{Q}[x_1,x_2,\dots, x_n] \mid 
				f(z)=0
				\text{ for all } 
				z \in V_{\mathbb R} (J) 
				 \right\}~.
	\end{align*}
\end{proposition}

\subsection{Problems with standard approaches to deciding ACR ideal theoretically} \label{sec:how-not-to-decide-ACR}
In this subsection, we give several sufficient conditions for ACR in terms of the steady-state ideal (Proposition~\ref{prop:sufficient-conditions-ACR-using-steady-state-ideal}). 
We show how this result can be useful (Example~\ref{ex:JoshiNguyen}). 
 However, we also 
show 
through examples that none of the sufficient conditions is necessary for ACR, and additionally that some other approaches to deciding ACR are also incomplete.

\begin{prop}[Sufficient conditions for ACR] \label{prop:sufficient-conditions-ACR-using-steady-state-ideal}
Let $I$ be the steady-state ideal of a mass-action system $( G, \kappa)$ with  $n$ species.  Let $i \in \{1,2,\dots, n\}$.  
Assume one of the following holds:
	\begin{enumerate}
		\item there exists $\alpha \in \mathbb{R}_{>0}$ such that $x_i-\alpha $ is in $I$,
		\item there exists $\alpha \in \mathbb{R}_{>0}$ such that $x_i-\alpha$ is in the saturation $ I:(x_1 x_2 \dots x_n)^{\infty}$,
		\item there exists $g \in I \cap \mathbb{R}[x_i]$ such that $g$ has a unique positive root.
	\end{enumerate}
Then 
 $( G, \kappa)$ 
 has ACR in species $X_i$.
\end{prop}

\begin{proof}
Condition (1) is a special case of (2) (and also (3)), so we begin with (2).  Assume that  
$x_i-\alpha$ is in $ I:(x_1 x_2 \dots x_n)$, for some $\alpha>0$.  This means that 
$(x_1 x_2 \dots x_n)^m (x_i-\alpha) \in I$ for some $m \geq 0$.  Now assume that $x^*$ is a positive steady state. Then $(x^*_1 x^*_2 \dots x^*_n)^m (x^*_i-\alpha) =0$ and so (as $x_j^*>0$ for all $j$) we have $x^*_i=\alpha$.  Hence, there is ACR in $X_i$ with ACR-value $\alpha$.  

For (3), assume that $g \in I \cap \mathbb{R}[x_i]$ has a unique positive root $\alpha$.  Let $x^*$ be a positive steady state.  Then $g \in I$ implies that $g(x^*_i)=0$ and so $x^*_i=\alpha$ since $\alpha$ is the only positive root of $g$.  We conclude that $(G,\kappa)$ has ACR in species $X_i$ (with ACR-value $\alpha$).
\end{proof}

The following example 
shows how to use condition (2) in Proposition~\ref{prop:sufficient-conditions-ACR-using-steady-state-ideal} to detect ACR.

\begin{example}
\label{ex:JoshiNguyen}
Motivated by bifunctional enzymes in signal transduction networks, Joshi and Nguyen 
introduced the network below and showed it has ACR in $S_3$ \cite[Table~1, line 6]{joshi2023bifunctional}:  
 \begin{equation}
\begin{split}
    S_1 + E
    \overset{\kappa_{1}}{\underset{\kappa_{2}}{\rightleftarrows}}
    C_1 
    \overset{\kappa_{3}}{\rightarrow}  
    S_2 + E
    \overset{\kappa_{4}}{\underset{\kappa_{5}}{\rightleftarrows}}
    C_2
    \overset{\kappa_{6}}{\rightarrow}  
    S_3 + E\\
    S_2 + C_1
    \overset{\kappa_{7}}{\underset{\kappa_{8}}{\rightleftarrows}}
    C_3 
    \overset{\kappa_{9}}{\rightarrow}  
    S_3 + C_1
    \overset{\kappa_{10}}{\underset{\kappa_{11}}{\rightleftarrows}}
    C_4
    \overset{\kappa_{12}}{\rightarrow}  
    S_1 + C_1 
  \label{eq:example_Joshi}
\end{split}
\end{equation}
The unusual ordering of the species $S_2,S_3,S_1$ in the second line of~\eqref{eq:example_Joshi} is deliberate.

Joshi and Nguyen showed that the network~\eqref{eq:example_Joshi} has ACR in $S_3$, by appealing to a result on bifunctional enzymes in futile cycles of a very specific form~\cite{joshi2023bifunctional}.  As for general-purpose results on ACR in the literature, they do not apply.  For instance, the results of~\cite{shinar2010structural} 
cannot be used, because the network has deficiency $2$.  Similarly, the results of~\cite{tonello2017network}
cannot be used, because the the network does not allow a `simple weakly reversible translation'.

We show next that Proposition~\ref{prop:sufficient-conditions-ACR-using-steady-state-ideal} applies to this network.  
Let $x_1, x_2, \dots, x_8$ denote the concentrations of the species
$S_1,S_2,S_3,E,C_1,C_2,C_3,C_4$, respectively.  In what follows, we use symbolic rate constants, although Proposition~\ref{prop:sufficient-conditions-ACR-using-steady-state-ideal} pertains to specialized values of rate constants, because the output we show below does not depend on the specific choice of the rate constants. 
The steady-state ideal $I$ is generated by the following eight polynomials:
\begin{multline*}
     -\kappa_{1}x_1x_4+\kappa_{2}x_5+\kappa_{12}x_8~,\, 
     	\kappa_{3}x_5+\kappa_{5}x_6-\kappa_{4}x_2x_4-\kappa_{7}x_2x_5+\kappa_{8}x_7~, \\  
    \kappa_{6}x_6+\kappa_{9}x_7-\kappa_{10}x_3x_5+\kappa_{11}x_8~,\, 
 -\kappa_{1}x_1x_4+(\kappa_{2}+\kappa_{3})x_5-\kappa_{4}x_2x_4+(\kappa_{5}+\kappa_{6})x_6~, \\
 \kappa_{1}x_1x_4-(\kappa_{2}+\kappa_{3})x_5-\kappa_{7}x_2x_5+(\kappa_{8}+\kappa_{9})x_7-\kappa_{10}x_3x_5+(\kappa_{11}+\kappa_{12})x_8~, \\
 \kappa_{4}x_2x_4-(\kappa_{5}+\kappa_{6})x_6~,\,
 \kappa_{7}x_2x_5-(\kappa_{8}+\kappa_{9})x_7~,\,  
 \kappa_{10}x_3x_5-(\kappa_{11}+\kappa_{12})x_8~.
\end{multline*}

For any choice of positive $\kappa_1,\dots, \kappa_{12}$, the saturation $ I:(x_1 x_2 \dotsm x_8)^{\infty}$ contains the polynomial $\kappa_{10}\kappa_{12}x_3-\kappa_3(\kappa_{11}+\kappa_{12})$.  Thus, by Proposition~\ref{prop:sufficient-conditions-ACR-using-steady-state-ideal}(2), there is ACR in $S_3$ with ACR-value $\kappa_3(\kappa_{11}+\kappa_{12})/\kappa_{10}\kappa_{12}$.
\end{example}

The next example shows that the sufficient conditions for ACR in Proposition~\ref{prop:sufficient-conditions-ACR-using-steady-state-ideal} are not necessary.  
(We remark that the authors -- and other researchers as well~\cite{neigenfind2011biochemical} -- once mistakenly believed condition~(1) in the proposition to be necessary for ACR!)

\begin{example}[Example~\ref{ex:ideals-network}, continued] \label{ex:wrong2} 
Recall that the network  $G=\{ B \overset{\kappa_1}\to A, ~ 2A+B \overset{\kappa_2}\to A+2B \}$ has ACR in species $A$.  
We consider fixed values $\kappa_i>0$ for the rate constants (the subsequent analysis does not depend on their precise values).  
The ODEs are ${dx_A}/{dt} = - {dx_B}/{dt} = x_B(\kappa_1 - \kappa_2 x_A^2)$.    
We see immediately that conditions~(1) and ~(3) do not hold, as the steady-state ideal contains no nonzero univariate polynomials.
After saturating the steady-state ideal by $x_A x_B$, we obtain $\langle \kappa_1 - \kappa_2 x_A^2 \rangle$. 
However, this ideal does not contain $x_A- \sqrt{\kappa_1/ \kappa_2}$, nor any other polynomial of the form $x_A - \alpha$ where $\alpha>0$.  
In other words, condition~(2) of  Proposition~\ref{prop:sufficient-conditions-ACR-using-steady-state-ideal} is violated. 

\end{example}

From Example \ref{ex:wrong2}, we might propose to factor $ \kappa_1 - \kappa_2 x_A^2$; indeed, $x_A- \sqrt{\kappa_1/ \kappa_2}$ is a zero-divisor of the steady-state ideal.    
One might therefore ask whether having a unique such zero-divisor is necessary or sufficient for ACR.  In other words, we consider the following condition (where the notation is as in Proposition~\ref{prop:sufficient-conditions-ACR-using-steady-state-ideal}):
\begin{align} \label{eq:wrong-condition}
	\textrm{there exists a unique } \alpha \in \mathbb{R}_{>0} \textrm{ such that } 
	x_i - \alpha \textrm{ is in } I \textrm{ or }
	x_i - \alpha \textrm{ is a zero-divisor of } I.
\end{align}
\noindent
This will motivate the definition of zero-divisor ACR appearing in Section~\ref{sec:specialized-GB}.
However, the next examples show that condition~\eqref{eq:wrong-condition} is neither necessary nor sufficient for ACR. 

\begin{example} \label{ex:zero-div-condition-not-necessary}
Consider the network
$G=\{  3A \overset{1}{\to} 4A,~ A+2B \overset{1}{\to} 2A+2B,~
	2A \overset{2}{\to} A,~ A+B \overset{4}{\to} B,~
	A \overset{5}{\to} 2A,~ 2A+B \overset{1}{\to} 2A+2B,~ 3B \overset{1}{\to} 4B,~
	A+B \overset{2}{\to} A,~ 2B \overset{4}{\to} B,~
	B \overset{5}{\to} 2B\}$.
The resulting mass-action ODE system is as follows:
\begin{align*}
  \frac{dx_A}{dt} & =~ x_A^3+x_Ax_B^2-2x_A^2-4x_Ax_B+5x_A ~=: f_A \\
  & =~ x_A[(x_A-1)^2+(x_B-2)^2]~,\\
  \frac{dx_B}{dt} & =~ x_A^2x_B+x_B^3-2x_Ax_B-4x_B^2+5x_B ~=: f_B \\
   & =~ x_B[(x_A-1)^2+(x_B-2)^2].
 \end{align*}
It follows that the only positive root of $f_A=f_B=0$ is $(x_A, x_B)=(1,2)$, so this system has ACR in both species.
However, $x_A-1$ is not a zero-divisor of the steady state ideal; this is easily checked computationally (as explained in Remark~\ref{rmk:check-zero-divisor}) or alternatively a direct proof can be given.  We conclude that condition~\eqref{eq:wrong-condition} is not necessary for ACR.
\end{example}

\begin{example} \label{ex:zero-div-condition-not-sufficient}
Consider the following network:
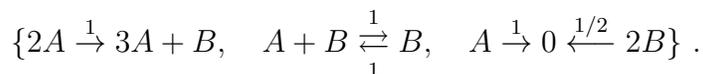
\[\{	2A  \overset{1}\to 3A + B, \quad 
	A + B \overset{1}{\underset{1}\rla} B, \quad 
	A \overset{1}\to 0 \overset{1/2}\longleftarrow 2B\}~.\]
The resulting mass-action ODEs admit 
the following factorization: 
$\frac{dx_A}{dt} = (x_A-1)(x_A-x_B)$ and $
	\frac{dx_B}{dt} =  (x_A-x_B)(x_A+x_B)$.
It can be proved that the polynomial $(x_A-1)$ is the unique zero-divisor of the steady-state ideal $I$ of the form $x_A - \alpha$. 
Nonetheless, there is no ACR. 
Indeed, the set of steady states is given by a line, $x_A=x_B$, and so $(1,1)$ and $(2,2)$ are positive steady states that differ in both coordinates.  
Hence, condition~\eqref{eq:wrong-condition} is not sufficient for ACR.
\end{example}

The previous 
examples pertain to networks with complexes that are beyond bimolecular (like $3A+B$).  
Indeed, Example~\ref{ex:zero-div-condition-not-necessary} involves degree-3 polynomials. 
On the other hand, Example~\ref{ex:zero-div-condition-not-sufficient} features degree-2 polynomials only (the trimolecular complex, $3A+B$, is not a reactant).  
A natural question is whether there are examples with lower 
molecularity:
\begin{question} \label{q:lower-degree-condition}
For networks that are at-most-bimolecular, is condition~\eqref{eq:wrong-condition} necessary for ACR?  Is it sufficient?
\end{question}

\noindent
Biochemical networks arising in applications are typically at-most-bimolecular, so it would be useful to have an answer to Question~\ref{q:lower-degree-condition}.

An alternative approach to deciding ACR, is to try sampling steady states from a few compatibility classes, but this too might not suffice.
\begin{example} \label{ex:wrong4}
Consider the following network:
 \begin{equation} \label{eq:network-2-components}
\{ 	2A+ B  \overset{1}\to 3A, \quad 
	A+B  \overset{3}\to 2B, \quad 
	B \overset{2}\to A\}~.
 \end{equation}
The ODEs factor as shown here:
\begin{align*}
	\frac{dx_A}{dt} ~&=~ x_B (x_A-1)(x_A-2) \\
	\frac{dx_B}{dt} ~&=~ - x_B (x_A-1)(x_A-2) ~.
\end{align*}
Hence, the set of steady states has two components, arising from $x_A=1$ and $x_A=2$.  These are depicted by  vertical, red lines below, and the stoichiometric compatibility classes are the dashed, blue lines:
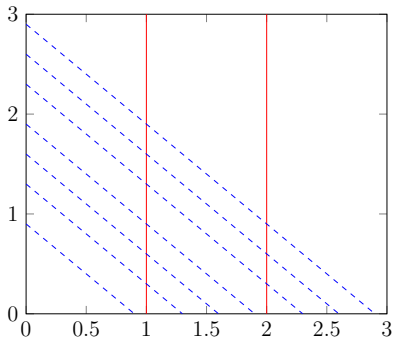
\begin{center}
\begin{tikzpicture}[scale = 0.7]
\begin{axis}[xmin=0,xmax=3,ymin=0,ymax=3]
\addplot[color=red] coordinates {
	(1, 0)
	(1,3)};
\addplot[color=red] coordinates {
	(2, 0)
	(2,3)};
\addplot +[mark=none,color=blue,dashed] coordinates {(0,0.9) (0.9,0)};
\addplot +[mark=none,color=blue,dashed] coordinates {(0,1.3) (1.3,0)};
\addplot +[mark=none,color=blue,dashed] coordinates {(0,1.6) (1.6,0)};
\addplot +[mark=none,color=blue,dashed] coordinates {(0,1.9) (1.9,0)};
\addplot +[mark=none,color=blue,dashed] coordinates {(0,2.3) (2.3,0)};
\addplot +[mark=none,color=blue,dashed] coordinates {(0,2.6) (2.6,0)};
\addplot +[mark=none,color=blue,dashed] coordinates {(0,2.9) (2.9,0)};
\end{axis}
\end{tikzpicture}
\end{center}
If we sampled a few stoichiometric compatibility classes, and checked whether the steady-state value of $x_A$ is always the same, we could be unlucky -- and conclude erroneously that $A$ has ACR with ACR-value $1$. 
 \end{example}

As before, Example~\ref{ex:wrong4} involves a degree-3 polynomial.  
So, we again would like to know whether there are similar examples involving lower-degree polynomials.

We end this section with an example of a network arising in biology that exhibits ACR.

\begin{example}[Shinar and Feinberg network] \label{ex:ShiFein}
The following network was analyzed (and was shown to have ACR) by Shinar and Feinberg  \cite[Figure~2(B)]{shinar2010structural} and has been studied by many others \cite{A-E-J, dexter2013dimerization, Enciso2016, karp2012complex, messi}:
 \begin{equation}
\begin{split}
\{
    X 
    \overset{\kappa_{1}}{\underset{\kappa_{2}}{\rightleftarrows}}
    XT
    \overset{\kappa_{3}}{\rightarrow}  
    X_p,
    \quad 
    %
    X_p + Y
    \overset{\kappa_{4}}{\underset{\kappa_{5}}{\rightleftarrows}}
    X_pY
    \overset{\kappa_{6}}{\rightarrow}
    X + Y_p,
    \quad 
    %
    XT + Y_p
    \overset{\kappa_{7}}{\underset{\kappa_{8}}{\rightleftarrows}}
    XTY_p
    \overset{\kappa_{9}}{\rightarrow}
    XT + Y 
 \}   ~.
  \label{eq:example_Feinberg_multisite}
\end{split}
\end{equation}
We denote by $x_1, x_2,\dots, x_7$ the concentrations of the species as follows:
$$x_{X}=x_1, \; x_{XT}=x_2, \; x_{X_p}=x_3, \; 
x_{Y}=x_4, \; x_{Y_p}=x_5, \; x_{X_pY}=x_6, \; x_{XTY_p}=x_7~.$$
The steady state ideal $I$ is generated by 
$-\kappa_{1}x_1+\kappa_{2}x_2+\kappa_{6}x_6,~ \kappa_{3}x_2-\kappa_{4}x_3x_4+\kappa_{5}x_6,~ \kappa_{6}x_6-\kappa_{7}x_2x_5+\kappa_{8}x_7,~ \kappa_{4}x_3x_4-(\kappa_{5}+\kappa_{6})x_6,~ \kappa_{7}x_2x_5-(\kappa_{8}+\kappa_{9})x_7$. 
For any choice of positive rate constants $\kappa_1, \kappa_2, \dots, \kappa_9$,
the reduced Gr\"{o}bner basis of $I$ with respect to the lexicographic order 
$x_1 >x_2 > x_3 > x_4 > x_6 > x_7 > x_5$ is 
\begin{align*}
 {\mathcal G} ~&=~ 
    \{\kappa_7\kappa_9x_5x_7-\kappa_3(\kappa_8+\kappa_9)x_7,~
    \kappa_6x_6-\kappa_9x_7,~
    \kappa_4\kappa_6x_3x_4-(\kappa_5+\kappa_6)\kappa_9x_7,~ 
    \\
    & \quad \quad \quad \quad 
    \kappa_3x_2-\kappa_9x_7,~
    \kappa_1\kappa_3x_1-(\kappa_2+\kappa_3)\kappa_9x_7 \}~.
\end{align*}

\noindent
Here, ACR in $X_5$ can be seen from the first Gr\"obner basis element, $g_1:= \kappa_7\kappa_9x_5x_7-\kappa_3(\kappa_8+\kappa_9)x_7$.  
Indeed, $g_1/x_7= \kappa_7\kappa_9x_5-\kappa_3(\kappa_8+\kappa_9) $ is a linear polynomial in $x_5$ and so, for all choices of rate constants, this network exhibits ACR in the fifth species, namely, $Y_p$.
\end{example}

\begin{remark} \label{rem:appendix-complex-number}
    One message of this subsection is that, in general, ACR is not readily detected from the steady-state ideal.  If, however, we allow for a more general form of ACR (complex-number ACR), then detection from the steady-state ideal is possible.  See Appendix~\ref{sec:cpx-acr-appendix}.
\end{remark}

\subsection{The real radical}

In this section, we consider the computational aspects of the problem of deciding whether a system $(G,\kappa)$ has non-vacuous ACR. 
We first explain why, in many of our results, we assume that the rate constants are rational (Section~\ref{sec:assm-rational}).  
Next, we show that, in some cases, we can use a new ideal, which we call the positive-restriction ideal 
(Section~\ref{sec:using-j-ideal}), and also decompositions of its radical ideal (Section~\ref{sec:ideal_decomp}) to assess ACR.

\subsubsection{On the assumption of rational-number rate constants} \label{sec:assm-rational}
We have seen that it is sometimes convenient to assume that the rate constants are rational numbers 
(as in \S~\ref{sec:1-species-algo}).
We make this assumption again for some results in this section.  

The reason for this assumption, in the previous section, was to avoid issues concerning whether and how we can compute over the real numbers.  
Another reason comes from the fact that, in real-life applications, we are interested in ACR (and other properties) that persist when rate constants are perturbed. In such situations, approximating real-number rate constants by rational numbers would suffice.

Nevertheless, in principle, 
the assumption of rational rate constants may be problematic. 
There are two main difficulties.  The first is the issue of approximating the (actual) real-number coefficients by rational ones. The second arises when the ACR-value is irrational, and then this value either is approximated (and so propagates numerical errors) or is described symbolically (but possibly in an uninformative way). 
Such difficulties are shown in the next example. 
Precisely, we face the problem of dealing numerically with irrational-number roots.

\begin{example}\label{ex:Q_2}
 Consider the network $G=\{ 0 \overset{2} \to A,~ 2A \overset{1}\to 0,~ 2B \overset{1}\to 3B,~ A+B \overset{1}\to A,~ B\overset{1.41}\to 2B\}$.
The resulting ODEs are as follows: 
 \begin{align*}
  \frac{dx_A}{dt} & =~ 2-x_A^2~,\\
  \frac{dx_B}{dt} & =~ x_B^2 - x_A x_B + 1.41 x_B  ~=~ x_B(x_B-(x_A-1.41))~.
 \end{align*}
The elimination ideal $I \cap \Q [x_A]$, where $I$ is the steady-state ideal, is generated by $2-x_A^2$.  Hence, this system has ACR with ACR-value $\sqrt{2}$.  However, numerical problems may arise when checking whether ACR is vacuous or non-vacuous.  
Specifically, when we try to find the possible steady-state values of $x_B$, we have the following issue:
if we approximate $x_A= \sqrt{2}$ by any rational number less than or equal to $1.41$, 
we do not obtain any positive steady states and so we would conclude that the system has vacuous ACR.
However, the system has non-vacuous ACR in both species, because the only positive steady state is $(x_A^*, x_B^*)=(\sqrt{2},\sqrt{2}-1.41)$.
\end{example}

\subsubsection{Assessing ACR using the positive-restriction ideal} \label{sec:using-j-ideal} 
As explained in the prior subsection, the results in this subsection pertain to systems with rational rate constants.  
Specifically, we show that, for such systems, ACR can be detected from the real radical of an ideal we call the positive-restriction ideal (Proposition~\ref{prop:ideal-extended}).
This ideal expands the steady-state ideal (and resides in a polynomial ring with additional variables $z_i$) so that the resulting real variety forms a ``cover'' of the original positive steady states (see Lemma~\ref{lem:j-ideal}). The purpose of this ideal, therefore, is to restrict our attention to the positive steady states (rather than including boundary ones), hence the name ``positive-restriction ideal''.

\begin{definition} \label{def:j-ideal}
Consider a mass-action system $(G,\kappa)$ with $n$ species and $r$ reactions, where  
$\kappa \in \mathbb{Q}_{>0}^r$.  
The {\em positive-restriction ideal}
of  $(G,\kappa)$ is the ideal in $\Q[x_1, x_2, \dots, x_n, z_1, z_2, \dots, z_n]$ generated by 
the right-hand sides of the ODEs in~\eqref{eq:ODE-mass-action}, $f_{\kappa}(x)_1, ~f_{\kappa}(x)_2, ~\dots, ~f_{\kappa}(x)_n$ (viewed here as polynomials in $\Q[x_1, x_2, \dots, x_n, z_1, z_2, \dots, z_n]$) 
and the polynomials of the form $x_iz_i^2-1$:
 \begin{align}\label{eq:J_G_kappa}
  J(G,\kappa)~:=~  \langle f_{\kappa}(x)_1, ~f_{\kappa}(x)_2, ~\dots, ~f_{\kappa}(x)_n, ~x_1z_1^2-1,~ x_2z_2^2-1,~ \dots ,~ x_nz_n^2-1 \rangle~.
  \end{align}
\end{definition}

\begin{remark} \label{rem:rabino-trick}
    In Definition~\ref{def:j-ideal}, the use of the polynomials $x_iz_i^2-1$ is a version of the `Rabinowitsch trick' that is used in the proof of Hilbert's Nullstellensatz.
\end{remark}

\begin{lemma} \label{lem:j-ideal}
For a reaction system $(G,\kappa)$ with $n$ species and $\kappa \in \mathbb{Q}_{>0}^r$ (where $r$ is the number of reactions), 
the following projection (to the first $n$ coordinates) is a $2^n$-to-$1$ surjection from the real variety of the positive-restriction ideal 
to the positive steady-state locus:
	\begin{align*}
	V_{\mathbb{R}} (J(G,\kappa)) & ~\overset{\phi}\to~ V_{>0}(I(G,\kappa)) \\
						( x;z) & ~\mapsto ~ x~.
	\end{align*}
\end{lemma}

\begin{proof}
We first show that the image of $\phi$ is contained in $V_{>0}(I(G,\kappa)) $.  
Assume that $(x;z) \in V_{\mathbb{R}} (J(G,\kappa))$.  Then $x_iz_i^2-1= 0$ for all $i$, so $z_i \neq 0$ and $x_i>0$ for all $i$.  The vector $x$ is also a root of the polynomials in $I(G,\kappa)$, since the generators of $I(G,\kappa)$ also belong to  $J(G,\kappa)$.  We conclude that $x \in V_{>0}(I(G,\kappa))$.

Next, let $x \in V_{>0}(I(G,\kappa))$. 
For all $i$, let $z_i := \pm \sqrt{1/x_i}$.  It follows that $(x;z) \in V_{\mathbb{R}} (J(G,\kappa))$.  
It is straightfoward to see that no other choices of $z_i$ yield $(x;z) \in V_{\mathbb{R}} (J(G,\kappa))$.  
So, $\phi$ is surjective and is  $2^n$-to-$1$.
\end{proof}

\begin{remark}[Viewing the positive-restriction ideal as a steady-state ideal]  \label{rmk:positive_radical}
For a 
mass-action system $(G, \kappa)$ with $n$ species, 
the positive-restriction ideal 
$J(G,\kappa)$
can be viewed as the steady-state ideal of the mass-action system obtained by adding $2n$ reactions to $(G,\kappa)$, namely, the reactions
$ 0 \overset{1}\to Z_i$ and $
X_i + 2Z_i \overset{1} \to X_i + Z_i$
for all $i= 1,2,\dots, n$.
\end{remark}

Next, we use Lemma~\ref{lem:j-ideal} and the Real Nullstellensatz (Proposition~\ref{prop:real-nullstellensatz}) to extend Proposition~\ref{prop:ideal}, in the case when the rate constants -- and also the ACR-value -- are rational. 

\begin{prop}[Rational-number ACR and ideals] \label{prop:ideal-extended}
Consider a mass-action system $( G, \kappa)$ with  $\kappa \in \mathbb{Q}_{>0}^r$, 
and let $\alpha \in \mathbb{Q}_{>0}$. The following are equivalent: 
\begin{enumerate}
\item $( G, \kappa)$ has ACR in species $X_i$ with ACR-value $\alpha$,
	\item $x_i - \alpha$ is in the ideal of the positive steady state locus of $(G,\kappa)$, 
	\item $x_i - \alpha$ 
	is in 
	  $\sqrt[\R]{J(G,\kappa)}$ (the real radical of the positive-restriction ideal). 
\end{enumerate}
\end{prop}

\begin{proof}
The equivalence of (1) and (2) is Proposition~\ref{prop:ideal}.
Now we prove that (1) and (3) are equivalent.  Let $\alpha >0$.  
By Lemma~\ref{lem:j-ideal}, every positive steady state $x^*$ satisfies $x_i^*= \alpha$ if and only if
 every $(x^*;z^*)$ in the real variety of the positive-restriction ideal 
 satisfies $x^*_i = \alpha$, which in turn 
 (by Proposition~\ref{prop:real-nullstellensatz}) 
  is equivalent to the condition $(x_i - \alpha) \in \sqrt[\R]{J(G,\kappa)}$.  This final equivalence we just asserted uses the hypothesis that $\alpha \in \mathbb{Q}_{>0}$; indeed, if instead the ACR-value $\alpha$ were irrational, then its irreducible polynomial (over $\mathbb Q$), rather than the linear polynomial $x_i - \alpha$, would be in $\sqrt[\R]{J(G,\kappa)}$.
\end{proof}

\begin{remark} \label{rmk:acr-value-rational}
In Proposition~\ref{prop:ideal-extended}, it is assumed that not only are the rate constants rational, but also the ACR-value itself.  This is not always the case (for instance, we saw in Example~\ref{ex:ideals-network} a network in which the ACR-value has the form
$ \sqrt{\kappa_1 / \kappa_2}$).  Nevertheless, this property of rational-number ACR-value is common in the literature, and can be explained in part by the theory of ``robust ratios'' which, when applicable, gives a precise rational-number expression of the ACR-value in terms of the rate constants~\cite{tonello2017network}.  
\end{remark}

\begin{remark} \label{rmk:acr-value-irrational}
When the ACR-value is irrational, Proposition~\ref{prop:ideal-extended} does not apply.  Nevertheless, we can see from the proof of the proposition that sometimes ACR can be detected.  Indeed, as in the end of the proof, if a univariate (in $x_i$), irreducible polynomial $g$ appears in the real radical of the 
positive-restriction ideal, 
and $g$ has a unique positive root $\alpha$, then there is ACR with ACR-value $\alpha$.
\end{remark}

Proposition~\ref{prop:ideal-extended} reveals an advantage of considering rational-number rate constants: {\em Detecting ACR reduces to being able to compute a real radical ideal -- as long as the ACR-value is also rational.  }

In fact, real radicals can be computed effectively to some extent 
\cite{Neuhaus1998,Spang2008}, although the complexity of such algorithms is high. 
Moreover, new methods for computing real radicals have been developed recently~\cite{BaldiMourrain2021, SAFEYELDIN2021259}.  We note, however, that when these algorithms encounter irrational roots $\alpha$ of polynomials, either a numerical approximation of $\alpha$ is returned, or, for symbolic algorithms, the irreducible polynomial of $\alpha$. 
Similar drawbacks were shown earlier in Example~\ref{ex:Q_2}.

\subsubsection{ACR via ideal decomposition} \label{sec:ideal_decomp}
This subsection analyzes ACR by decomposing the positive-restriction ideal $J(G,\kappa)$ and then, using the notion of non-singular zeros, by restricting to certain components of $J(G, \kappa)$ (see Theorem~\ref{thm:P1P2-extended}). We also discuss issues involved in turning Theorem~\ref{thm:P1P2-extended} into an algorithm for detecting ACR (when the ACR-value is rational). 

We first introduce notation we use in this subsection. 
Given an ideal $P$ of $\Q[x_1,x_2,\dots,x_n]$, we let $\widetilde P$ denote the ideal of $\R[x_1,x_2,\dots,x_n]$ generated by the elements of $P$:
\begin{equation}\label{eq:PinR}
 \widetilde{P} ~:=~ P \R[x_1,x_2,\dots,x_n].
\end{equation}

We now state a lemma that points in the direction we are pursuing.

\begin{lemma}\label{lem:Q1}
If $J$ is an ideal of the ring $R:= \mathbb{Q}[x_1,x_2,\dots,x_n]$ that can be decomposed as $J=Q_1\cap Q_2$, 
where $Q_1$ and $Q_2$ are ideals in $R$ such that $V_\R(Q_2)=\varnothing$, then $V_\R(J)=V_\R(Q_1)$.
\end{lemma}
\begin{proof}
 This result is immediate since $V_\R(J)=V_\R(Q_1)\cup V_\R(Q_2)$.
\end{proof}

Next, we recall a criterion that can help us isolate certain components of the positive-restriction ideal.  To state that result (Proposition~\ref{prop:nonsingular} below), we must first recall several definitions (see \cite{Basu2006Algorithms,real-algebraic-geometry-1998}). 
An ideal $I$ in a ring $R$ is {\em prime} if $xy\in I$ implies that $x\in I$ or $y\in I$. 
The \emph{dimension} of an ideal $I$ of $ \R[x_1,x_2,\dots,x_n]$ is the Krull dimension of the quotient ring $\R[x_1,x_2,\dots,x_n]/I$. If $I$ is generated by $n-d$ algebraically independent polynomials, then the dimension of $I$ is $d$ (see~\cite[Corollary~3.7]{Kunz1991IntroductionTC}). 

We now adapt to our context the definition of a non-singular zero, and subsequently recall a result that appears in the book of Bochnak, Coste, and 
Roy~\cite[Proposition~3.3.16]{real-algebraic-geometry-1998}.  

\begin{definition} \label{def:non-singular}
Let  $P = \langle g_1, g_2, \dots, g_t  \rangle$ be a prime ideal of $\R[x_1, x_2, \dots, x_n]$ of dimension~$d$. A point $\mathbf{x}\in V_\R(P)$ is a \emph{non-singular zero} 
of $P$ if $\mathrm{rank}\left(\left[\frac{\partial g_i}{\partial x_j}(\mathbf{x})\right]_{i,j} \right)=n-d$.
\end{definition}

\begin{proposition}\label{prop:nonsingular}
 Let $P$  be a prime ideal of $\R[x_1, x_2, \dots, x_n]$. 
  If $P$ has a non-singular zero, then $P=\sqrt[\R]{P}$.
\end{proposition}

From this proposition, we deduce the following result. 

\begin{theorem} \label{thm:P1P2-extended}
 Consider a mass-action system $(G,\kappa)$ with $\kappa\in\Q_{>0}^r$, and let $J(G,\kappa)$ be the positive-restriction ideal~\eqref{eq:J_G_kappa}.  Let $\alpha\in \Q_{>0}$.
 Assume the following:
	\begin{enumerate}
	\item \label{it:1} $J(G,\kappa) = Q_1\cap Q_2$ for some ideals $Q_1,Q_2$ of $ \Q[x_1, x_2, \dots, x_n,z_1,z_2,\dots,z_n]$ such that $V_\R(Q_2)=\varnothing$, and 
	\item  \label{it:2} $Q_1= \overset{\ell}{\underset{i=1}{\bigcap}}P_i$ 
	for some ideals $P_i$ of  $\Q[x_1, x_2, \dots, x_n,z_1,z_2,\dots,z_n]$ such that $V_{\R}(P_i)\neq \varnothing$  (for all $i=1, 2, \dots, \ell$).
\end{enumerate}
Let $\widetilde{P}_i$, for all $i=1, 2, \dots, \ell$, denote the corresponding ideals in $\R[x_1, x_2, \dots, x_n,z_1,z_2,\dots,z_n]$,
as in~\eqref{eq:PinR}. 
Assume two additional \textbf{main hypotheses}: 
\begin{enumerate}
  \setcounter{enumi}{2}
   \item \label{it:impossible} $\widetilde{P}_i$ is a prime ideal of $\R[x_1, x_2, \dots, x_n,z_1,z_2,\dots,z_n]$ (for all $i=1, \dots \ell$), and
 \item \label{it:4} $\widetilde{P}_i$ has a non-singular zero (for all $i=1,2,\dots, \ell$).
\end{enumerate}
Then $(G,\kappa)$ has ACR in some species $X_j$ with ACR-value $\alpha$ if and only if $x_j-\alpha \in Q_1$.
\end{theorem}

\begin{proof}
Let $J:=J(G,\kappa)$.
By Proposition~\ref{prop:ideal-extended}, it suffices to prove the equality $Q_1 = \sqrt[\R]{J} $. 

Hypothesis~\eqref{it:1} and
  Lemma~\ref{lem:Q1} together imply that $V_\R(J)=V_\R(Q_1)$.
  Hence, by the Real Nullstellensatz (Proposition~\ref{prop:real-nullstellensatz}), we obtain the first equality here:
  \begin{align} \label{eq:radical-ideals}  
  \sqrt[\R]{J}~=~
	\sqrt[\R]{Q_1}
	~=~
  	\overset{\ell}{\underset{i=1}{\bigcap}} \sqrt[\R]{P_i}~,
\end{align}
and the second equality follows from hypothesis~\eqref{it:2} and the fact that the real radical of an intersection is the intersection of the real radicals \cite[Lemma~2.2]{Neuhaus1998}.  
Next, equation~\eqref{eq:radical-ideals} and the containment $Q_1 \subseteq \sqrt[\R]{Q_1}$, together imply that $Q_1 \subseteq \sqrt[\R]{J}$.  Hence, using equation~\eqref{eq:radical-ideals}, it remains only to show the containment $\left( \overset{\ell}{\underset{i=1}{\bigcap}} \sqrt[\R]{P_i } \right) \subseteq  Q_1$.

To verify this containment, first note that $\sqrt[\R]{P_i}\subseteq \sqrt[\R]{\widetilde{P}_i}\cap \Q[x_1,x_2,\dots,x_n,z_1,z_2,\dots,z_n]$ (for all $i=1, 2, \dots, \ell$).  
Next, by hypotheses~(\ref{it:impossible}--\ref{it:4}) and Proposition~\ref{prop:nonsingular}, we have $\sqrt[\R]{\widetilde{P}_i}=\widetilde{P}_i$. Finally, $\widetilde{P_i}\cap\Q[x_1,x_2,\dots,x_n,z_1,z_2,\dots,z_n]=P_i$~\cite[Theorem~7.5]{Matsumura1989CommutativeRT}. 
We conclude that $\left( \overset{\ell}{\underset{i=1}{\bigcap}} \sqrt[\R]{P_i } \right)  \subseteq \overset{\ell}{\underset{i=1}{\bigcap}}P_i  =Q_1 $ (the equality is by hypothesis~(\ref{it:2})), so  the desired containment holds.
\end{proof}

The hypotheses of Theorem~\ref{thm:P1P2-extended} are, unfortunately, not easily checked.  Indeed, in general, we cannot detect whether an ideal is prime over $\mathbb{R}$ or whether it has a non-singular zero.  Nevertheless, when such computations are possible, the theorem tells us that if certain ideals $\widetilde{P}_i$ arising from a decomposition of $Q_1$ are  prime over $\mathbb{R}$ and admit non-singular zeros, ACR with rational ACR-value is characterized by the presence of a linear polynomial $x_i - \alpha$ in $Q_1$.  (See Steps 1 and 2 in the procedure below.)

But perhaps the most interesting case arises precisely when some of the ideals $\widetilde{P}_i$ \uline{lack} non-singular zeros.
In this case,  if $\widetilde{P}_i=\langle g_1, g_2, \dots,g_t\rangle$ has dimension $d$ and has no non-singular zeros, then the $(n-d)\times(n-d)$ minors of the Jacobian matrix of $g_1, g_2, \dots,g_t$ are polynomials in $x_1, x_2, \dots,x_n$ that vanish at all the real zeros of $J$. Accordingly, we can add such polynomials to $J$  -- without affecting its real variety -- and instead analyze this larger ideal.  This idea underlies the following procedure  (in particular, Steps 3 and 4).

\begin{proc} \label{prod:ideal-decomposition}
Detecting ACR by ideal decomposition
\end{proc}

\begin{itemize}
 \item[Input.] An ideal $Q_1$ of  $ \Q[x_1, x_2, \dots, x_n,z_1,z_2,\dots,z_n]$, such that $V_\R(Q_1)=V_\R(J)$, where $J:=J(G,\kappa)$ is the  positive-restriction ideal of a mass-action system $(G,\kappa)$.
 \item[Output.] ``ACR''  or  ``Inconclusive''.
 \end{itemize}

\begin{itemize}
 \item[Step~1.] 
 	Decompose $Q_1$ as $Q_1=\overset{\ell}{\underset{i=1}{\bigcap}}P_i$, where (for all $i=1, 2, \dots, \ell$) $P_i$ is an ideal of the ring $ \Q[x_1, x_2, \dots, x_n, z_1, z_2, \dots, z_n]$ and $\widetilde{P}_i$, as in~\eqref{eq:PinR}, is a prime ideal of the ring $\R[x_1,x_2,\dots,x_n,z_1,z_2,\dots,z_n]$.

 \item[Step~2.] If $\widetilde{P}_i$ has a non-singular zero for all $i=1,2,\dots, \ell$, then check whether $Q_1$ contains a polynomial of the form $x_i-\alpha$, where $\alpha \in \mathbb{Q}_{>0}$. If so, then output, ``ACR in $X_i$ with ACR-value $\alpha$''. 
 \item[Step~3.] Otherwise, for every $P_i=\langle g_1^i, g_2^i, \dots,g_{t_i}^i\rangle$ of dimension $d_i$ with no non-singular zeros, add all the size-$(2n-d_i)$ minors of the Jacobian matrix of $(g_1^i, g_2^i, \dots,g_{t_i}^i)$ to $Q_1$. Call this new ideal $Q_1^{(1)}$.
 \item[Step~4.] Repeat, \emph{if possible}, Steps~1--3 for the ideal $Q_1=Q_1^{(1)}$.
\end{itemize}

By repeating Steps~1--4, we can \emph{in theory} generate an ascending chain of ideals $Q_1\subseteq Q_1^{(1)}\subseteq Q_1^{(2)}\subseteq \dots \subseteq  Q_1^{(j)}\subseteq \cdots$ that stabilizes, since the ring $\Q[x_1,x_2,\dots,x_n,z_1,z_2,\dots,z_n]$ is Noetherian. 
The advantage of working with the ideal $Q_1^{(j)}$ instead of $J$ is that its dimension is smaller (and yet it has the same real variety).

For a better understanding of this theoretical procedure, and how to apply it to compute and detect ACR (when all steps can be done effectively), consider the following example:

\begin{example}\label{ex:singular}

Recall the system from Example~\ref{ex:zero-div-condition-not-necessary}, which generates the following ODEs:
\begin{align*}
 \frac{dx_A}{dt}  &= x_A[(x_A-1)^2+(x_B-2)^2]\\
 \frac{dx_B}{dt} &= x_B[(x_A-1)^2+(x_B-2)^2].
\end{align*}

It is easy to see in this example that the system shows ACR for both variables, since the only positive solution is $(x_A,x_B)=(1,2)$.
%
However, we will ignore this obvious fact and instead apply the above procedure. 
Unless otherwise noted, all computations below can be checked using a computer algebra system. 

The positive-restriction ideal, $J=\langle f_A,~ f_B,~ x_Az_A^2-1,~x_Bz_B^2-1 \rangle$ in $\mathbb{Q}[x_A, x_B, z_A, z_B]$, can be decomposed as follows:
$$J=Q_1 \cap Q_2 \cap Q_3 \cap Q_4~,$$
where 
\begin{align*}
	Q_1 &= \langle  (x_A-1)^2+(x_B-2)^2, ~x_Az_A^2-1,~x_Bz_B^2-1 \rangle, \\ 
	Q_2 &=\langle x_A, ~(x_A-1)^2+(x_B-2)^2,~ x_Az_A^2-1,~x_Bz_B^2-1 \rangle, \\
	Q_3 &= \langle x_B, ~(x_A-1)^2+(x_B-2)^2, ~x_Az_A^2-1,~x_Bz_B^2-1\rangle, \, {\rm and} \\
	Q_4 &= \langle  x_A,~x_B, ~x_Az_A^2-1,~x_Bz_B^2-1 \rangle~. 
\end{align*}

We easily see that  $V_{\C}(Q_2)=V_{\C}(Q_3)=V_{\C}(Q_4)=\varnothing$. 
So, Lemma~\ref{lem:Q1} implies that $V_{\R}({J})=V_{\R}(Q_1)$. This ideal $Q_1$ does not  contain a polynomial of the form $x_A-\alpha$ or $x_B-\alpha$ (where $\alpha>0$), so we apply Procedure~\ref{prod:ideal-decomposition} to $Q_1$.  
 
It is straightforward to check that $\widetilde Q_1$ is a prime ideal in $\R[x_A, x_B,z_A,z_B]$ with dimension~1. Hence, Step~1 is accomplished with $\ell=1$ and $P_1 = Q_1$. 

Next, for Step~2, we must check whether $\widetilde Q_1$ has a nonsingular zero.
Accordingly, we consider the matrix of partial derivatives of the generators of $\widetilde Q_1$ (namely, $g_1:=x_A^2-2x_A+x_B^2-4x_B+5$, $g_2:=x_Az_A^2-1$, $g_3:=x_Bz_B^2-1$): 
$$M~=~ 
	\left[\begin{array}{cccc}
          2x_A-2 & 2x_B-4 & 0 & 0\\
          z_A^2 & 0 & 2x_Az_A & 0\\
          0 & z_B^2 & 0 & 2x_Bz_B\\
         \end{array}\right].$$
A real zero of $\widetilde Q_1$ is non-singular if the matrix $M$ -- when specialized at that zero -- has rank~3, that is, one of the following  $3 \times 3$ minors of $M$ is nonzero:
	\begin{align*}
	h_1 &:= -4(x_A-1)x_A z_Az_B^2~, \quad \quad 
	& h_2 &:= -4(x_B-2) x_B  z_A^2 z_B~, \\
	h_3 &:= 8(x_A-1) x_A x_B z_A z_B~, \quad \quad 
	& h_4 &:= 8(x_B-2) x_A x_B z_A z_B~.
	\end{align*}
It follows that $\widetilde Q_1$ has a nonsingular zero if and only if at least one of the following ideals has nonempty real variety (for $i=1,2,3,4$):
\[
	I_i := \langle g_1, g_2, g_3, h_i w -1 \rangle \quad \subseteq \quad \mathbb{R}[x_A, x_B, z_A, z_B, w]~,
\]
where $w$ is a new variable (we refer the reader to \cite[Chapter~13]{Basu2006Algorithms}). 
Indeed, we find that  $V_{\R}(I_i)=\varnothing$ for all $i=1,2,3,4$, and so $\widetilde Q_1$ has {\em no} nonsingular zero.

We proceed to Step~3.  We define the ideal $Q_1^{(1)}=\langle g_1,g_2,g_3,h_1,h_2,h_3,h_4\rangle$, whose real variety then coincides with the real variety of $J$, and has dimension 0 ($<1$). We again apply Steps~1 and~2, this time for $Q_1^{(1)}$, as follows. 
This ideal can be decomposed as $Q_1^{(1)}=P_1^{(1)}\cap P_2^{(1)}$, where $P_1^{(1)}, P_2^{(1)}$ are the following two prime ideals of $\R[x_A, x_B, z_A, z_B]$: 
\[
 P_1^{(1)}= \langle  x_A-1, ~x_B-2,~ z_A-1,~2z_B^2-1 \rangle, \quad  P_2^{(1)}= \langle  x_A-1, ~x_B-2,~ z_A+1,~2z_B^2-1 \rangle.
\]
As $z_B\neq0$ for any zero of $P_i^{(1)}$, it is straightforward to check that any zero of $P_i^{(1)}$ is non-singular, for $i=1,2$. As both $x_A-1$ and $x_B-2$ belong to $Q_1^{(1)}$ we conclude that the system has ACR in both species, with ACR-values $1$ and $2$, respectively.
\end{example}

\section{Detecting zero-divisor ACR}  \label{sec:specialized-GB} 
In this section, we present a procedure for detecting ACR that works well for a type of ACR we call ``zero-divisor ACR''.  We motivate and define this concept in Section~\ref{sec:motivation}, and then present our procedure (Algorithm~\ref{Algorithm}) -- and examples of its usage -- in Section~\ref{sec:algorithm}.  The ideas behind this algorithm come from the theory of Gr\"obner bases for ideals involving parameters, which we describe in Section~\ref{sec:connection-theory}.

\subsection{Gr\"obner bases, elimination orders, and zero-divisor ACR} \label{sec:motivation}
Fix $i \in \{1,2,\dots, n\}$, and call $\hat{\mathbf{x}}_i=(x_1, \dots, x_{i-1},x_{i+1},\dots,x_n)$. An {\em elimination order} for $\hat{\mathbf{x}}_i$ on $\Q[x_1, x_2, \dots, x_n]$ is a monomial order such that every polynomial with leading monomial in $\Q[x_i]$ belongs to $\Q[x_i]$. 
One instance of such an order is a lexicographic order with $x_i$ smaller than the rest of the variables. Another example is a product order where 
\[
 \hat{\mathbf{x}}_i^\alpha x_i^a \succ \hat{\mathbf{x}}_i^\beta x_i^b \; 
 	\quad \Longleftrightarrow  \quad \; 
	\hat{\mathbf{x}}_i^\alpha \succ \hat{\mathbf{x}}_i^\beta, \text{ or } \hat{\mathbf{x}}_i^\alpha = \hat{\mathbf{x}}_i^\beta \text{ and } a>b.
\]
Given an elimination order $\succ$ for $\hat{\mathbf{x}}_i$ on $\Q[x_1,x_2,\dots,x_n]$, let $\succ_{\hat{\mathbf{x}}_i}$ denote the order on the ring $\Q[x_i][\hat{\mathbf{x}}_i]$ obtained by restricting the order $\succ$ to the monomials in $\Q[\hat{\mathbf{x}}_i]$. 
We denote the resulting {\em leading coefficient} of a polynomial $f\in \mathbb{Q}[x_i][\hat{\mathbf{x}}_i]$ as follows:
\[
\mathrm{lc}_{\hat{\mathbf{x}}_i}(f)~\in~ \mathbb{Q}[x_i]~.
\]

One fact that we use below is the well-known {\em Elimination Theorem}.  This theorem states that for any elimination order for $\hat{\mathbf{x}}_i$ on $\Q[x_1,x_2,\dots,x_n]$, if we compute a Gr\"obner basis $\mathcal{G}$ of an ideal $I\subseteq \Q[x_1,x_2,\dots,x_n]$,  then $I\cap\Q[x_i]\neq\{0\}$ if and only if $\mathcal{G}\cap\Q[x_i]\neq\varnothing$, and, in this case, $I\cap\Q[x_i]=\langle \mathcal{G}\cap\Q[x_i]\rangle$ (see, for instance, \cite[Theorem~4.8]{Hassett}).

Our aim is to use Gr\"obner bases to propose candidates for ACR species and their corresponding ACR-values. We motivate this approach through the following example.

\begin{example}\label{ex:algorithm}
 Consider the following network and specified rate constants:
$G=\{  2A+C \overset{1}{\to} 2A+2C,~ A+C \overset{3}{\to} A,~ 
	C \overset{2}{\to} 2C,~ 2A+D \overset{1}{\to} 2A+2D,~ 
	A+D \overset{4}{\to} A,~ D \overset{3}{\to} 2D,~ C \overset{1}{\to} B,~ 
	B \overset{1}{\to} 0,~ A+C+D \overset{1}{\to} C+D,~ 
	C+D \overset{1}{\to} A+C+D\}$. 
The steady-state ideal $I$ is generated by 
	$f_A=-x_Ax_Cx_D+x_Cx_D$, 
	$f_B=x_C-x_B$, 
	$f_C=x_A^2x_C-3x_Ax_C+2x_C$, and 
	$f_D=x_A^2x_D-4x_Ax_D+3x_D$.

Consider the lexicographic order $x_B > x_C > x_D > x_A$. The reduced Gr\"obner basis of $I$ with respect to this order is 
 \begin{equation}\label{eq:G_A}
  {\mathcal G} = \{x_A^2x_D-4x_Ax_D+3x_D,~x_A^2x_C-3x_Ax_C+2x_C,~ x_Ax_Cx_D-x_Cx_D,~ x_B-x_C\}~.
 \end{equation}
As ${\mathcal G}\cap\Q[x_A]=\varnothing$, the Elimination Theorem implies that that $I\cap \Q[x_A]=\{0\}$.
The leading coefficients with respect to $\hat{\mathbf{x}}_A=(x_B, x_C, x_D)$ are, respectively: 
	\begin{align*}  
	x_A^2-4x_A+3, \quad x_A^2-3x_A+2,\quad x_A-1,\quad 1~.
 	\end{align*}
 	
From the third element of the Gr\"obner basis~$\mathcal G$, 
namely, $x_Ax_Cx_D-x_Cx_D = x_C x_D (x_A - 1)$, we see that neither $ x_C x_D $ nor $x_A - 1$ is in $I$, and so  $x_A-1$ is a zero-divisor of $I$ (recall Definition~\ref{def:zero-divisor}). Moreover, it is easy to deduce from this polynomial that, in every positive steady state, the value of $x_A$ must be $1$.  That is, there is ACR in $X_A$ with ACR-value $1$. 
\end{example}

Example~\ref{ex:algorithm} inspires the next definition and the algorithm in the next subsection.

\begin{definition} \label{def:zero-divisor-acr}
Assume $\alpha >0$.  A mass-action system $(G,\kappa^*)$ has {\em zero-divisor ACR} in $X_i$ with ACR-value $\alpha$ if $(G,\kappa^*)$ has ACR in $X_i$ with ACR-value $\alpha$ and either (1) $x_i - \alpha$ is in the steady-state ideal or (2) $x_i - \alpha$ is a zero-divisor of the steady-state ideal.
\end{definition}

\begin{example}[Example~\ref{ex:ShiFein}, continued] \label{ex:ShiFein1.5}
We return to the Shinar and Feinberg network~\eqref{eq:example_Feinberg_multisite}. 
From the Gr\"obner basis shown earlier, we see that the linear polynomial 
 $g_1/x_7= \kappa_7\kappa_9x_5-(\kappa_3\kappa_8+\kappa_3\kappa_9) $ is a zero-divisor of $I$, 
 as
 $g_1/x_7 \notin I$, 
 $x_7\notin I$, and $g_1\in I$.  Therefore, for all choices of rate constants, this network exhibits zero-divisor ACR (in the fifth species, namely, $Y_p$).
\end{example}

\begin{remark}
The concept of zero-divisor ACR is closely related to condition~\eqref{eq:wrong-condition}, which we saw earlier is neither necessary nor sufficient for ACR.  In particular, we saw a mass-action system with ACR, but not zero-divisor ACR (Example~\ref{ex:zero-div-condition-not-necessary}).  We also saw a situation where we have zero-divisor ACR, and saturating the steady-state ideal was not enough to make $x_i - \alpha$, where $\alpha$ is the ACR-value, appear in the ideal (Example~\ref{ex:wrong2}).
\end{remark}

\begin{remark}[Zero-divisor ACR detection through Gr\"obner bases] \label{rem:gb}
For many mass-action systems (or even networks), zero-divisor ACR can be detected from a 
reduced 
Gr\"obner basis $\mathcal G$ under a suitable elimination order for $\hat{\mathbf{x}}_i$.  More precisely, in many examples  in the literature, when there is zero-divisor ACR in some species $X_i$ with ACR-value $\alpha$, the following holds:
\begin{align} \label{eq:wrong-condition-gb}
	\textrm{there exists } \alpha \in \mathbb{R}_{>0} \textrm{ such that } 
	x_i - \alpha \textrm{ divides some element of } \mathcal G~.
\end{align}
Such a system was shown in Example~\ref{ex:algorithm}.  As for a network, recall, in Examples~\ref{ex:ShiFein} and~\ref{ex:ShiFein1.5}, that $\kappa_7\kappa_9x_5-(\kappa_3\kappa_8+\kappa_3\kappa_9)$ divides the first element of the reduced Gr\"{o}bner basis of the steady-state ideal $I$ with respect to the lexicographic order $x_1 >x_2 >  x_3 > x_4 > x_6 > x_7 > x_5$.  Thus, in this example, ACR in $X_5$ is detected from (factors of) the Gr\"obner basis, as in~\eqref{eq:wrong-condition-gb}.  
\end{remark}

In light of Remark~\ref{rem:gb}, it is natural to ask whether condition~\eqref{eq:wrong-condition-gb} is necessary for zero-divisor ACR.
  However, this is not true, as we see below in Example~\ref{ex:nonstraightforward_zero_divisor}. This fact motivates the need for Algorithm~\ref{Algorithm} in the next subsection.

\begin{example}[Zero-divisor ACR is not detected through Gr\"obner bases] \label{ex:nonstraightforward_zero_divisor}
 Consider the following polynomials:
 \begin{align*}
  f_1 ~&=~ x_Bx_C[(x_A-1)^2(x_B+3)+(x_C-2)] \\
  f_2~&=~x_Bx_C[(x_A-1)(x_B+3)^2-(x_C-2)]~.
 \end{align*}
These equations lead to a mass-action system in three species $A$, $B$, and $C$ with $dx_A/dt=f_1$, $dx_B/dt=dx_C/dt=f_2$ (recall Remark~\ref{rmk:hungarian-lemma}). 
The reduced Gr\"{o}bner basis of the steady-state ideal $I= \langle f_1, f_2 \rangle$ with respect to the lexicographic order $x_B >x_C > x_A$ (which is an elimination order for $\{x_B,x_C\}$, that is, for eliminating $x_A$) is ${\mathcal G}=\{g_1,g_2,g_3,g_4\}$, where
 \begin{align}
  	\notag
  g_1 &= x_Bx_C(x_C-2)(-x_C+x_A^3-3x_A^2+3x_A+1)~, \\
	\label{eq:gb-for-nonstraightforward}
  g_2 &=x_Bx_C(x_Bx_A^2-2x_Bx_A+x_B+x_C+3x_A^2-6x_A+1)~, \\
  	\notag
  g_3 &= x_Bx_C(-x_Bx_C+3x_Bx_A^2-6x_Bx_A+5x_B-x_Cx_A+x_C+9x_A^2-16x_A+7)~, \\
  	\notag
  g_4 &=x_Bx_C(x_B^2x_A-x_B^2+6x_Bx_A-6x_B-x_C+9x_A-7)~.
 \end{align}
Observe that ${\mathcal G} \cap \Q[x_A]=\varnothing$ (and so $x_A-1$ is not in $I$) and also that $x_A-1$ does not divide any element of ${\mathcal G}$.

 To confirm that this example shows that 
 condition~\eqref{eq:wrong-condition-gb} is not necessary for zero-divisor ACR, 
  it suffices to show that there is zero-divisor ACR in $A$ with ACR-value $1$. 
 To see that $x_A-1$ is a zero-divisor of $I$, observe that the following polynomial belongs to $I= \langle f_1, f_2 \rangle$:
 \[
  h~=~f_1+f_2~=~(x_A-1)x_Bx_C(x_B+3)[(x_A-1)+(x_B+3)]~.
 \]
We saw earlier that $x_A-1$ is not in $I$, and it is straightforward to check that $h/(x_A-1)= x_Bx_C(x_B+3)[(x_A-1)+(x_B+3)]$ also is not in $I$.  Hence, $x_A-1$ is a zero-divisor of $I$.

Next, from examining the factors of $h$, we see that every positive steady state satisfies $x_A=1$.  We use this fact, together with the structure of the polynomials $f_1$ and $f_2$, to see that the positive steady-state locus is $\{(1,x_B,2) \mid x_B >0\}$.  This confirms that there is zero-divisor ACR in species $A$ with ACR-value $1$ (and also ACR in species  $C$).
\end{example}

\subsection{Algorithm} \label{sec:algorithm}
The ideas in the prior subsection motivate Algorithm~\ref{Algorithm}, which outputs candidates for ACR species (and the corresponding ACR-values). Once we obtain a candidate ACR-value $\alpha$ for species $X_i$, we can compute the colon ideal $I:(x_i-\alpha)$ to recover  a polynomial $h$ such that the product $(x_i-\alpha)h\in I$ (see Example~\ref{ex:nonstraightforward_zero_divisor_cont}).

\begin{algorithm}[!htb] \KwIn{A mass-action system $(G,\kappa^*)$ with $n$ species (and rate constants $\kappa^*$ in $\mathbb{Q}$).}
\KwOut{
A finite set $S$ of pairs 
$(X_j, \beta)$. }
 Initialize $S=\varnothing$\;
 Let $I$ be the steady-state ideal of $(G,\kappa^*)$\;
 \For{$i\in\{1,2,\dots,n\}$}{
 	Call $\hat{\mathbf{x}}_i=(x_1,\dots,x_{i-1},x_{i+1},\dots,x_n)$. Let $\mathcal G$ be a reduced Gr\"obner basis of $I$ for an elimination order for $\hat{\mathbf{x}}_i$ on $\Q[x_1,x_2,\dots,x_n]$.
	Let ${\mathcal G}_{i}:={\mathcal G}\cap {\mathbb Q}[x_i]$\;
  \eIf{${\mathcal G}_{i} \neq \varnothing$}{Let $\widetilde h$ denote the unique element of ${\mathcal G}_{i}$\;
	\For{$\alpha$ a positive root of $\widetilde h$}{append $(X_i,\alpha)$ to $S$}
   }{\For{$g \in \mathcal{G}$}{
   \For{$\alpha$ a positive root of the leading coefficient ${\rm lc}_{\hat{\mathbf{x}}_i}(g)$}{append $(X_i,\alpha)$ to $S$}}   
  }
 }
    \caption{{\bf Candidates for ACR} \label{Algorithm}}
\end{algorithm}

We conjecture, as follows, that Algorithm \ref{Algorithm} succeeds for all systems with zero-divisor ACR. 

\begin{conjecture} \label{conj:algorithm}
Consider a  mass-action system $(G,\kappa^*)$ with rational-number rate constants. 
If $(G, \kappa^*)$ has zero-divisor ACR in species $X_i$ with ACR-value $\alpha$,
then  $(X_i, \alpha)$ is one of the outputs of Algorithm~\ref{Algorithm}.
\end{conjecture} 

\begin{remark} \label{rem:easy-case-of-conj}
Recall from Definition~\ref{def:zero-divisor-acr} that zero-divisor ACR means that either (1) $x_i - \alpha$ is in the steady-state ideal or (2) $x_i - \alpha$ is a zero-divisor of the steady-state ideal.  In case~(1), it is straightforward to see that Conjecture~\ref{conj:algorithm} holds.  Indeed, in that case, 
$x_i - \alpha$ is the unique element of ${\mathcal G}\cap {\mathbb Q}[x_i]$, and so $(X_i,\alpha)$ is an output of the algorithm.
\end{remark}

Next, we present examples that show how to apply Algorithm~\ref{Algorithm} and interpret its output.

\begin{example}[Example~\ref{ex:algorithm}, continued]\label{ex:algorithm_cont}
 Recall that the network in Example~\ref{ex:algorithm} has steady-state ideal 
 $I=\langle-x_Ax_Cx_D+x_Cx_D,~ x_C-x_B,~ x_A^2x_C-3x_Ax_C+2x_C,~ x_A^2x_D-4x_Ax_D+3x_D\rangle$. 
\begin{itemize}
 \item The reduced Gr\"obner basis $\mathcal G$ of $I$ with respect to  the lexicographic order $x_B > x_C > x_D > x_A$ was shown in~\eqref{eq:G_A}, where we obtained  ${\mathcal G}\cap\Q[x_A]=\varnothing$. Now Algorithm~\ref{Algorithm} asks us to find the positive roots of the leading coefficients with respect to $\mathbf{x}=(x_B, x_C, x_D)$; these coefficients are $x_A^2-4x_A+3, x_A^2-3x_A+2, x_A-1, 1$. Therefore, the algorithm outputs three pairs: $(x_A,1)$, $(x_A,3)$, and $(x_A,2)$.
 \item The reduced Gr\"obner basis of $I$ with respect to the lexicographic order $x_A > x_C > x_D > x_B$ is 
 \[
  {\mathcal G} = \{x_C-x_B,~ x_Ax_Bx_D-x_Cx_D,~ x_A^2x_B-3x_Ax_C+2x_C,~x_A^2x_D-4x_Ax_D+3x_D\}.
 \]
 As ${\mathcal G}\cap\Q[x_B]=\varnothing$, the algorithm asks us to find the positive roots of the leading coefficients with respect to $\mathbf{x}=(x_A, x_C, x_D)$: $1$ and $x_B$. As there are no such positive roots,  the algorithm returns no candidates for ACR-values in $x_B$. 
 	\item The situation for $x_C$ and $x_D$ is similar to $x_B$: no candidate ACR-values are returned.
\end{itemize}
Thus, for this system,  Algorithm~\ref{Algorithm} returns three possible ACR-values for species $A$, namely, $1$, $2$, and $3$.  Next, we show that, in fact, there is ACR in $A$, and the ACR-value is one of the candidates, namely, $1$. 
Indeed, it is straightforward to see that the positive steady-state locus is 
$\{ (1,\alpha,\alpha,\beta) \mid \alpha>0,~ \beta>0\}$.  
As we saw in Example~\ref{ex:algorithm},  $x_A-1$ is a zero-divisor of $I$ and hence the system has zero-divisor ACR in the species $A$.
\end{example}

In the prior example, as we perform Algorithm~\ref{Algorithm}, we encounter the empty set for ${\mathcal G}_{i}:={\mathcal G}\cap {\mathbb Q}[x_i] $.  In contrast, this set is nonempty in the next example.

\begin{example} \label{ex:apply-algorithm-2-outputs} 
 Consider the network $G=\{  2A+C \overset{1}{\to} 2A+2C,~ A+C \overset{3}{\to} A,~ 
	C \overset{2}{\to} 2C,~ C \overset{1}{\to} B,~ 
	B \overset{1}{\to} 0,~ A \overset{4}{\to} 0,~ 0 \overset{3}{\to} A,~ 
	2A \overset{1}{\to} 3A\}$. 
When we apply Algorithm~\ref{Algorithm}, we obtain that the reduced Gr\"obner basis with respect to the lexicographic order $x_B>x_C>x_A$ of the steady-state ideal contains the polynomial $x_A^2-4x_A+3\in\Q[x_A]$. Thus, the pairs $(x_A,1)$ and $(x_A,3)$ are given as outputs. With arguments similar to those in Example~\ref{ex:algorithm}, we confirm there is ACR in species $A$ with ACR-value $1$.
\end{example}

The next example features a network with vacuous ACR.

\begin{example} \label{ex:network-no-pos-steady-states}
Consider the network $G=\{  A+C \overset{1}{\to} 2C,~ C \overset{1}{\to} A,~ C \overset{1}{\to} B+C,~  B \overset{1}{\to} 2B\}$.
The reduced Gr\"obner basis with respect to the lexicographic order $x_B>x_C>x_A$ of the steady-state ideal is ${\mathcal G}=\{x_Ax_C-x_C,~x_B+x_C\}$. 
Algorithm~\ref{Algorithm} reveals $x_A=1$ as a candidate for ACR-value, but this system has no positive steady states.
\end{example}

\begin{example}[Example~\ref{ex:ShiFein1.5}, continued] \label{ex:ShiFein2}
We return to the Shinar and Feinberg network~\eqref{eq:example_Feinberg_multisite}. 
Symbolic rate constants appear in what follows, even though the algorithm applies to specialized values of rate constants, because the steps of the algorithm (in this case) do not depend on the specific choice of the rate constants. 
We apply the algorithm, as follows.  
Recall that the reduced Gr\"{o}bner basis $ {\mathcal G}$ of the steady-state ideal $I$ with respect to the lexicographic order $x_1 >x_2 > x_4 > x_6 > x_7 > x_5$ satisfies ${\mathcal G}\cap \Q[x_5]=\varnothing$.
We
therefore consider the positive roots of the leading coefficients with respect to 
$\mathbf{x}=(x_1,x_2,x_3,x_4,x_6,x_7)$, which are 
$\kappa_7\kappa_9x_5-(\kappa_3\kappa_8+\kappa_3\kappa_9), \kappa_6, \kappa_4, \kappa_3, \kappa_1$. We obtain $(X_5,\frac{\kappa_3\kappa_8+\kappa_3\kappa_9}{\kappa_7\kappa_9})$ as an output of Algorithm~\ref{Algorithm}. As we saw earlier, indeed $X_5= Y_p$ exhibits (zero-divisor) ACR.
\end{example}

Next, we revisit Example~\ref{ex:nonstraightforward_zero_divisor} in which zero-divisor ACR was not detected by the Gr\"obner basis of the steady-state ideal.  Fortunately, however, ACR is detected through Algorithm~\ref{Algorithm}.

\begin{example}[Example~\ref{ex:nonstraightforward_zero_divisor}, continued]\label{ex:nonstraightforward_zero_divisor_cont}
We apply Algorithm~\ref{Algorithm} to the mass-action system in Example~\ref{ex:nonstraightforward_zero_divisor}. 
The leading coefficients of the Gr\"{o}bner basis elements~\eqref{eq:gb-for-nonstraightforward} with respect to $\mathbf{x}=(x_B, x_C)$ are $1$, $(x_A-1)^2$, $1$, and $x_A-1$, respectively. The algorithm therefore outputs the pair $(x_A,1)$.  (Recall that, indeed, there is ACR in species $A$.)

Next, we consider species $C$.  In the reduced Gr\"{o}bner basis of $I$ with respect to the lexicographic order $x_A >x_B > x_C$, the  leading coefficients with respect to $\mathbf{x}=(x_A, x_B)$ are $x_C (x_C -2)$ and $x_C$. Hence, the algorithm returns the pair $(x_C,2)$, which is the actual ACR-value for species $C$, as we saw in Example~\ref{ex:nonstraightforward_zero_divisor}. In fact, $x_C-2$ is also a zero-divisor of $I$, which is straightforward to check from the first element of the Gr\"{o}bner basis~\eqref{eq:gb-for-nonstraightforward}. 

Note that once we know that, for example, the ACR value for $x_A$ could be $\alpha=1$, then we can compute the colon ideal $I:(x_A-1)$ to recover $h=(x_A-1)x_Bx_C(x_B+3)[(x_A-1)+(x_B+3)]$. From this, we see that there is zero-divisor ACR in species $x_A$ with ACR-value $1$.

\end{example}

\subsection{Connection to parametric ideals and Gr\"obner covers} \label{sec:connection-theory}
In the future, we would like to resolve Conjecture~\ref{conj:algorithm}.  One approach to doing so is through the theory of Gr\"obner bases for ideals involving parameters, which in fact inspired Algorithm~\ref{Algorithm} in the first place.  We describe this theory, as it pertains to the algorithm, in this subsection.

In general, a parametric ideal is an ideal 
$I \subseteq \mathbb{F}(t_1, t_2, \dots, t_m)[x_1, x_2, \dots, x_n]$, where $t_1, t_2, \dots, t_m$ are the parameters. Given a Gr\"obner basis $\mathcal{G}$ of $I$, a fundamental problem is to find a proper variety $W\subseteq \mathbb{F}^m$ such that $\mathcal{G}$ remains a Gr\"obner basis under all specializations $(t_1, t_2, \dots, t_m) \longmapsto (a_1, a_2, \dots, a_m) \in \mathbb{F}^m \smallsetminus W$.  Much theory has been developed to tackle this problem (see, for instance, the book~\cite{MontesGC}).
A common approach involves turning the parameters into variables, that is, considering the ring 
$\mathbb{F}[\mathbf{x},\mathbf{t}] = \mathbb{F}[x_1, x_2, \dots, x_n, t_1, t_2, \dots, t_m]$. 

In this setting, a result of Suzuki and Sato is fundamental~\cite{suzuki-sato2003} (see also~\cite[Lemma~3.11]{MontesGC} and \cite[\S 6.3 Proposition 1]{CLO15}).  We note that their result holds for an arbitrary (finite) number of parameters $t_1, t_2, \dots, t_m$. However, to simplify the exposition, we restrict this result 
to the one-parameter case and rephrase it in  our  context which is when this parameter is in fact one of the variables $x_i$. 
The resulting 
Lemma~\ref{lem:suzuki-sato} inspired our Algorithm~\ref{Algorithm}.

\begin{lemma}[Suzuki-Sato]\label{lem:suzuki-sato}
 Given an ideal $I\subseteq\Q[x_1,x_2,\dots,x_n]$, fix $i \in \{1,2,\dots, n\}$ and call $\hat{\mathbf{x}}_i=(x_1,\dots,x_{i-1},x_{i+1},\dots,x_n)$. Consider an elimination order $\succ$ for $\hat{\mathbf{x}}_i$ on $\Q[x_1,x_2,\dots,x_n]$, and let 
 $\mathcal{G}=\{g_1,g_2,\dots,g_s\}$ be a reduced Gr\"obner basis of $I$ with respect to this order. If $\alpha\in\overline{\Q}$ is such that $lc_{\hat{\mathbf{x}}_i} (g_j)|_{x_i=\alpha}\neq0$ for any $1\leq j\leq s$ with $g_j\notin \mathcal{G} \cap\Q[x_i]$, then the specialized set 
    $\{ g_1|_{x_i=\alpha},g_2|_{x_i=\alpha}, \dots, g_t|_{x_i=\alpha}\}$ 
    is a reduced Gr\"obner basis of the specialized ideal $I|_{x_i=\alpha}\subseteq\overline{\Q}[\hat{\mathbf{x}}_i]$ with respect to the order $\succ_{\hat{\mathbf{x}}_i}$.
\end{lemma}

\begin{remark} \label{rem:s-sato}
Lemma~\ref{lem:suzuki-sato}
 extends to any computable field $\mathbb{F}$, in which case we consider $\alpha$ in the algebraic closure of $\mathbb{F}$ \cite{MontesGC}. 
\end{remark}

As explained above, 
Lemma~\ref{lem:suzuki-sato} is part of the theory of parametric Gr\"obner bases.  
There are several algorithms for computing with parametric Gr\"obner bases, one of which is implemented in {\tt Singular}~\cite{DGPS} and is based on the theory of Gr\"obner covers~\cite{MontesGC, Grobcov}. We end this section by illustrating this theory, as it pertains to ACR, through the following example.

 \begin{example}[Example \ref{ex:nonstraightforward_zero_divisor_cont} continued; zero-divisor ACR detected via Gr\"obner covers] \label{ex:nonstraightforward_zero_divisor-2}
We revisit the mass-action system from Example \ref{ex:nonstraightforward_zero_divisor}, in which the steady-state ideal is generated by the  polynomials
$  f_1 = x_Bx_C[(x_A-1)^2(x_B+3)+(x_C-2)]$ and $
  f_2= x_Bx_C[(x_A-1)(x_B+3)^2-(x_C-2)]$.
The following {\tt Singular} code computes the Gr\"obner cover of $I =  \langle f_1, f_2 \rangle$: 

\begin{verbatim}
LIB "grobcov.lib";
ring r = (0,xA),(xB,xC),dp;
ideal I = xB*xC*((xA-1)^2*(xB+3)+(xC-2)), xB*xC*((xA-1)*(xB+3)^2+(xC-2));
grobcov(I,("rep",1));
\end{verbatim}


We suppress the lengthy output, but describe its structure and meaning, as follows.
The output of {\tt grobcov} is a list of lists (in this example, the list has two elements).  Each entry of the main list is a triple  of the form $\{\text{initial}, \text{basis}, \text{segment}\}$, where:
\begin{itemize}
    \item ``initial'' is a monomial ideal which is the initial ideal of any specialization at any point in  ``segment'', and 
    \item ``basis'' is the corresponding reduced  Gr\"obner basis. 
\end{itemize}
The option (``rep'',1) returns the segments in canonical C-representation, that is, as pairs of the form $(J_1,J_2)$ representing the interval $V(J_1) \smallsetminus V(J_2)$, where the ideals $J_1$ and $J_2$ are radical and $J_2 \supseteq J_1$.

In our case, the output shows that any specialization of the ideal $I$ at any point in the segment $V(0)\smallsetminus V(x_A - 1) = \Q \smallsetminus \{1\}$ has the initial ideal $\langle x_B^2x_C,\, x_Bx_C^3 \rangle$. However, when $x_A = 1$, the initial ideal is $\langle x_Bx_C^2  \rangle$. Recall from Example \ref{ex:nonstraightforward_zero_divisor} that $x_A=1$ is precisely the ACR-value for the corresponding mass-action system. Thus, in this example, the presence of ACR is detected via Gr\"obner covers.
\end{example}

Example~\ref{ex:nonstraightforward_zero_divisor-2} motivates the following question: {\em If a mass-action system has zero-divisor ACR, is this always detected via a Gr\"obner cover?}  We believe that an affirmative answer would be a key step toward resolving Conjecture~\ref{conj:algorithm}, and therefore validating Algorithm~\ref{Algorithm}.

\section{Numerical methods for detecting or precluding ACR}  

\label{sec:numerical}

In Section~\ref{sec:ACR}, we gave necessary and sufficient conditions for ACR in a reaction system $(G, \kappa)$ in terms of the positive steady-state locus $\sqrt[>0]{I(G,\kappa)}$ (recall Proposition \ref{prop:ideal}). 
As discussed throughout, the fundamental challenge to detecting ACR is that it requires an understanding of the ideal $\sqrt[>0]{I(G,\kappa)}$.  While a few
(restricted)
techniques exist to fully understand  $\sqrt[>0]{I(G,\kappa)}$, especially when working over $\mathbb R$ rather than $\mathbb Q$, we can gain some knowledge using techniques from numerical algebraic geometry.  

Numerical algebraic geometry is concerned with computing solutions to a system of polynomial equations, often using methods based on polynomial homotopy continuation (the fundamentals of this theory can be found in \cite{bates2013numerically}). 
Algorithms in numerical algebraic geometry 
aim to find isolated solutions to systems, and also the positive-dimensional solution sets, e.g. curves and surfaces. In essence, these algorithms give a geometric way to decompose a complex variety, and with care, can yield insight into the decomposition of a real variety.

In this section, we start by providing background on real algebraic sets (Section~\ref{sec:real-algebraic-sets}) and describing some tools and objects from numerical algebraic geometry, in particular, numerical irreducible decompositions and witness sets (Section~\ref{sec:witness-sets}).  We then 
use these tools to 
give an algorithm for detecting ACR (Procedure~\ref{proc:numerical-irreducible-decomposition}) and an algorithm for precluding ACR (Procedure~\ref{proc:denseset}). The first uses numerical irreducible decompositions and witness sets to see whether any component of $V_{\mathbb C}(I(G, \kappa))$ lies in a hyperplane of the form $\{ x_i - \alpha\}$ (Section~\ref{sec:numerical-irreducible-decomposition}).  The second uses tools for numerically solving polynomial equations to find a dense set of points on each component of $V_{\mathbb R}(I(G, \kappa))$ that intersects a specified region of the positive orthant (Section~\ref{subsec:real-connected-components}). In this section, we only consider ACR for systems $(G, \kappa)$ as opposed to networks. In particular, we view $\kappa$ as fixed. 

\begin{remark}
    In this section, we work with the steady-state ideal $I(G, \kappa)$, rather than the positive-restriction ideal $J(G, \kappa)$ (from Definition~\ref{def:j-ideal}).  This choice allows us to keep the number of equations and variables to half of those in the positive-restriction ideal.  Nevertheless, both of the algorithms we propose in this section can be used with $J(G, \kappa)$.
\end{remark}

\subsection{Real algebraic sets} \label{sec:real-algebraic-sets}

We start this section by introducing real algebraic sets and restating Proposition \ref{prop:ideal} from a geometric point-of-view (see Proposition~\ref{prop:hyperplane}). 

A \emph{real algebraic set}, or \emph{real algebraic variety}, is the intersection of a complex variety with $\mathbb R^n$. For instance, consider the real variety 
$$V_{\mathbb R} (\sqrt[>0]{I(G, \kappa)}) ~=~ \{ x \in \mathbb R^n \ | \ h(x) = 0 \text{ for all } h \in \sqrt[>0]{I(G,\kappa)}\}~.$$
We can rewrite this variety as the following intersection:
$$V_{\mathbb R} (\sqrt[>0]{I(G, \kappa)}) ~=~ \mathbb R^n \cap V_{\mathbb C}(\sqrt[>0]{I(G,\kappa)})~,$$
where 
$$V_{\mathbb C}(\sqrt[>0]{I(G, \kappa)}) ~=~ \{x \in \mathbb C^n \ | \ h(x) = 0 \text{ for all } h \in \sqrt[>0]{I(G, \kappa)}\}~.$$

\begin{proposition} \label{prop:hyperplane} 
A reaction system $(G, \kappa)$ has ACR in $X_i$ if and only if 
there exists $\alpha > 0$ such that
$V_{\mathbb R} (\sqrt[>0]{I(G, \kappa)})$ lies in 
the
hyperplane $\{x_i - \alpha=0\}$.
\end{proposition}

\begin{proof}  This result follows immediately from Proposition \ref{prop:ideal}.
\end{proof}
As we explain in the next subsection, every complex variety can be written as the union of irreducible components. 
Real algebraic sets can also be decomposed. Indeed, a real algebraic set $V_{\mathbb R}$ can be decomposed into the union of finitely many path-connected sets $C_1,C_2, \ldots, C_{\ell}$ where, for each $i$, both $C_i$ and $V_{\mathbb R} \smallsetminus C_{i}$ are closed in the Euclidean topology.  Each such set $C_i$  is a \emph{real connected component} of $V_{\mathbb R}$ (for more details, see \cite{bochnak2013real}).  The following result pertains to connected components of $V_{\mathbb R} (\sqrt[>0]{I(G, \kappa)})$, and it follows immediately from Proposition \ref{prop:hyperplane}.

\begin{corollary} \label{cor:ACR-real-components}
A reaction system $(G, \kappa)$ has ACR in $X_i$ if and only if 
 there exists $\alpha > 0$ such that
every real connected component of $V_{\mathbb R} (\sqrt[>0]{I(G, \kappa)})$ lies in the hyperplane $\{x_i - \alpha=0\}$.
\end{corollary}

Recall from~\eqref{eq:ideal-containments} the containments  $I(G,\kappa) 
~\subseteq~
\sqrt{I(G,\kappa)}
~\subseteq~
\sqrt[{> 0}]{I(G,\kappa)}$.  It follows that 
$$V_{\mathbb C}(\sqrt[{> 0}]{I(G,\kappa)}) \cap \mathbb R^n ~
\subseteq ~
V_{\mathbb C}(\sqrt{I(G,\kappa)}) \cap \mathbb R^n
~ \subseteq
~V_{\mathbb C}(I(G,\kappa) ) \cap \mathbb R^n,
$$
and so each real connected component of $V_{\mathbb R}(\sqrt[{> 0}]{I(G,\kappa)})$ is contained in a real connected component of $ V_{\mathbb R}(I(G,\kappa) )$.  
This fact, together with Corollary~\ref{cor:ACR-real-components}, yields the following sufficient condition to detect ACR.

\begin{proposition} \label{prop:real-components} 
If 
 there exists $\alpha > 0$ such that
every real connected component of $V_{\mathbb R} (I(G, \kappa))$ lies either in some coordinate hyperplane $\{x_j = 0\}$ or in the hyperplane $\{x_i - \alpha=0\}$, then the reaction system $(G, \kappa)$ has ACR in $X_i$. 
\end{proposition}

The condition in Proposition \ref{prop:real-components} is sufficient, but not necessary for ACR. For example, if $I(G,\kappa)=\langle x_Ax_B(x_A-1)(x_B+2)\rangle$, then  $V_{\mathbb R} (I(G, \kappa))$ has four real components: the real hyperplanes $\{x_A = 0\}$, $\{x_B = 0\}$, $\{x_A -1 = 0\}$, and $\{x_B + 2=0\}$. 
By examining the last component,  $\{x_B + 2=0\}$, we see that 
the reaction system $(G,\kappa)$ does not satisfy the hypothesis of Proposition \ref{prop:real-components}, despite having
 ACR in species A. 

\begin{remark} \label{rem:hyperplane-check}
    Proposition \ref{prop:real-components} is related to Remark \ref{rem:local-acar}, as the method of Pascual-Escudero and Feliu in \cite{beatriz-elisenda} aims to determine whether every real connected component of $V_{\mathbb R} (\sqrt[>0]{I(G, \kappa)})$ is contained in a finite union of hyperplanes of the form $\{x_i - \alpha = 0 \}$.
\end{remark}

In this section, we give two numerical techniques to check for ACR.  
Procedure \ref{proc:denseset} -- for precluding ACR -- relies on Proposition~\ref{prop:real-components}, while 
Procedure \ref{proc:numerical-irreducible-decomposition} -- for detecting ACR -- relies on numerical irreducible decompositions, which we define in the next subsection.

\subsection{Numerical irreducible decomposition and witness sets} \label{sec:witness-sets}
Let $V_{\C}$ be a complex variety.  The 
{\em pure $j$-dimensional component} of $V_{\C}$ is the union of the irreducible components of dimension $j$.  Thus, the variety $V_{\C}$ has an irreducible decomposition of the form
\begin{equation}\label{Eq:IrredDecomp}
V_{\C} ~=~ \bigcup_{j=0}^{\text{dim}V_{\C}} V_j ~=~ \bigcup_{j=0}^{\text{dim} V_{\C}} \bigcup_{l=1}^{k_j} V_{j,l}~,
\end{equation}
where $V_j$ is the pure $j$-dimensional component of $V_{\C}$, and each $V_{j,l}$ is a $j$-dimensional irreducible component, which can be found using symbolic or numerical methods.  Symbolic methods, which are more costly, find irreducible decompositions algebraically, and the output is a set of ideals described by generators.  In contrast, methods using numerical algebraic geometry find {\em numerical irreducible decompositions} by 
combining tools from algebraic geometry and numerical analysis (e.g., via a polynomial homotopy continuation), and the output is a \emph{witness set}. 

Here we briefly describe numerical irreducible decompositions (see~\cite{bates2013numerically} for details).  At a high level, each such decomposition is a collection of points obtained from intersecting each complex component of the variety with an affine generic linear space. 
In particular, let $V_{j,l}$ be a $j$-dimensional irreducible component of $V_{\C}$ with $d_{j,l} = \deg~V_{j,l}$; then for a fixed $j$-codimensional generic affine linear space $H_{j,l} \subseteq \C^n$, the intersection $V_{j,l} \cap H_{j,l}$ consists of $d_{j,l}$ points.  In this context, we will use the genericity condition in \cite[\S13.2]{sommese2005numerical}. In particular, an affine linear space $H$ of codimension $j$ defined by a linear system $L = (L_1, L_2, \ldots, L_j) = 0$ is \emph{generic with respect to an irreducible algebraic set $X \subseteq \mathbb C^n$} if for any given subset  $\{ L_{i_1} , L_{i_2} , \ldots, L_{i_r} \} \subseteq \{L_1, L_2, \ldots, L_j\}$ it follows that (1) either $X \cap  V_{\mathbb C}(L_{i_1} ,L_{i_2} , \ldots, L_{i_r})$ is empty or $\dim \left( X \cap  V_{\mathbb C}(L_{i_1} , L_{i_2} , \ldots, L_{i_r}) \right) = \dim X - r \geq 0$, and (2) the singular points of 
$X \cap  V_{\mathbb C}(L_{i_1} , L_{i_2} , \ldots, L_{i_r})$ are a subset of the singular points of $X$. 
 In the case of an algebraic set $X$ that is reducible,
an affine linear space $H$ defined by a linear system $L$ is generic with respect to $X$
if it
is generic with respect to all irreducible components of $X$, plus all irreducible components of intersections of any number of the irreducible components of $X$.

After associating to each $H_{j,l}$ a linear system $L_{j,l}$, such
that $V_{\C}(L_{j,l}) = H_{j,l}$.  The set $V_{j,l}\cap V_{\C}(L_{j,l}) = V_{j,l}\cap H_{j,l}$ is called a
{\em witness point set} for $V_{j,l}$. The triple
$\mathcal{W}_{j,l} = \{I(V_{\C}),~L_{j,l},~V_{j,l}\cap V_{\C}(L_{j,l})\}$ 
is a {\em witness set} for $V_{j,l}$. 
A {\em numerical irreducible decomposition} of $V_{\C}$ is of the form
\begin{equation}\label{Eq:NumIrredDecomp}
\bigcup_{j=0}^{\text{dim}V_{\C}} \bigcup_{l=1}^{k_j} \mathcal{W}_{j,l}~,
\end{equation}
where $\mathcal{W}_{j,l}$ is a witness set for a distinct $j$-dimensional irreducible component of $V_{\C}$ and the union operation is a formal union.
By construction, each witness set includes information about the dimension and degree of the corresponding component.

Next, if we change our choice of 
generic linear space 
$V_{\C}(L_{j,l})$,
then the set of points in $V_{j,l}\cap V_{\C}(L_{j,l})$ will change, unless $j=0$ (that is, unless $V_{j,l}$ is a 0-dimensional component). We state this observation formally, as follows.

\begin{proposition}\label{prop:two-linear-spaces}  

Let $V_{\mathbb C} \subseteq \mathbb C^n$ be an irreducible complex variety of dimension $d>0$. 
Assume that there exist
 $\alpha \in \mathbb{C}$, $i \in \{1,2, \ldots, n\}$, and 
 an affine codimension-$d$ linear space $H_1$
 that is generic with respect to $V_{\mathbb C}$ such that $p_i = \alpha$ for all $p \in V_{\mathbb C} \cap H_1$.  
Assume, additionally, that there exists
an affine codimension-$d$ linear space 
 $H_2$ 
 that is generic with respect to both $V_{\mathbb C}$ and $V_{\mathbb C} \cap \{x_i - \alpha =0 \}$,  such that $q_i = \alpha$ for all $q \in V_{\mathbb C} \cap H_2$. Then $V_{\mathbb C}$ lies in the hyperplane
 $\{ x_i - \alpha =0 \}$. 
\end{proposition}

\begin{proof} 
Assume $H_1$ is a generic affine codimension-$d$ linear space with respect to $V_{\mathbb C}$, and $H_2$ is a generic affine codimension-$d$ linear space with respect to both $V_{\mathbb C}$ and $V_{\mathbb C} \cap \{x_i - \alpha =0\}$. 
Now assume $p_i = q_i = \alpha$ for all $p \in V_{\mathbb C} \cap H_1$ and $q \in V_{\mathbb C} \cap H_2$.  By the genericity of $H_1$ and $H_2$, the sets $V_{\mathbb C} \cap H_1$ 
 and $V_{\mathbb C} \cap H_2$ are both nonempty 
zero-dimensional sets containing the same number of points \cite[Theorem 13.2.1]{sommese2005numerical}, in particular, $V_{\mathbb C} \ \cap H_2 \ \cap \ \{x_i - \alpha=0\}$ is nonempty.

Now, suppose $V_{\mathbb C}$ does \emph{not} lie in the hyperplane $\{x_i - \alpha=0\}$. This implies that \\ $\dim \left( V_{\mathbb C} \cap \{x_i - \alpha=0\} \right) \leq d-1$, and thus, by the genericity of $H_2$, the set $V_{\mathbb C} \ \cap H_2 \ \cap \ \{x_i - \alpha=0\}$ is empty, a contradiction. 
\end{proof}


\begin{remark} \label{remark:prob-one}
The hypothesis in Proposition \ref{prop:two-linear-spaces} would be satisfied with probability-one by choosing affine codimension-$d$ linear spaces $H_1$ and $H_2$ that are generic with respect to $V_{\mathbb C}$.
\end{remark} 




        The next subsection shows how to use numerical irreducible decompositions to detect ACR.

\begin{remark}
    Numerical irreducible decompositions can be computed using software such as {\tt Bertini} \cite{BHSW06}, {\tt PHCpack} \cite{verschelde1999algorithm}, and {\tt HomotopyContinuation.jl} \cite{breiding2018homotopycontinuation}.  
\end{remark}

\subsection{ACR via numerical irreducible decomposition} \label{sec:numerical-irreducible-decomposition}

Consider the complex variety of the steady-state ideal of a mass-action system $(G, \kappa)$ with $n$ species: 
$$V_{\C}^{G,\kappa} ~=~ V(I(G,\kappa)) ~=~  
    \{x \in \C^n \mid f_{\kappa}(x)_1 = f_{\kappa}(x)_1 = \dots f_{\kappa}(x)_n= 0
    \}~.$$


It follows from Proposition \ref{prop:two-linear-spaces} that we can detect ACR by computing two (distinct) numerical irreducible decompositions and then comparing results.  
This observation underlies the next procedure.  As the procedure requires checking equality up to an estimation error $\delta$, we refer to the ACR detected by the procedure as \emph{Numerical ACR within $\delta$}. Note though, due to estimation error, it is possible that the procedure falsely detects ACR.  

\begin{proc} \label{proc:numerical-irreducible-decomposition}
 Detecting Numerical ACR by numerical irreducible decomposition
 \end{proc}
\begin{itemize} 
    \item[Input.] A reaction system $(G, \kappa^*)$ with ideal $I(G, \kappa^*)$ generated by steady-state equations $f_{\kappa^*}(x)~=~ (f_{\kappa^*}(x)_1, f_{\kappa^*}(x)_2, \dots, f_{\kappa^*}(x)_n)$.  Numerical tolerance $\delta>0$.  
    \item[Output.] ``Numerical ACR within $\delta$''  or  ``Inconclusive.''
 \end{itemize}
 
\begin{itemize}
 \item[Step~1.] 
 	Compute two distinct numerical irreducible decompositions of $V_{\C}^{G,\kappa} $ and store witness sets in two lists, $\mathcal W$ and $\mathcal W'$.	
 \item[Step~2.] Remove witness sets that represent boundary components from $\mathcal W$ and $\mathcal W'$. In particular, for each witness set $\mathcal{W}_{j,l} = \{I(V_{\C}),L_{j,l},V_{j,l}\cap V_{\C}(L_{j,l})\}$, if there exists $1 \leq i \leq n$ such that for all points $(x_1,x_2, \ldots, x_n) \in V_{j,l}\cap V_{\C}(L_{j,l})$ it is the case that $|x_i|< \delta$, then update $\mathcal W$ by removing $\mathcal{W}_{j,l}$. Repeat the process for $\mathcal W'$. 
 \item[Step~3.]  If there exists a coordinate $x_i$ such that for all the points in $V_{j,l}\cap V_{\C}(L_{j,l})$, ranging over all $\mathcal W_{j,l} \in \mathcal W$, and for all the points in $V_{j,l}\cap V_{\C}(L'_{j,l})$, ranging over all $\mathcal W'_{j,l} \in \mathcal W'$, the absolute value of pairwise differences of the $i$th coordinates is less than $\delta$, then return ``Numerical ACR  in $X_i$ within $\delta$."
	\item[Step~4.] Otherwise, return ``Inconclusive."
 
\end{itemize}

\begin{proof}[Proof of correctness of Procedure~\ref{proc:numerical-irreducible-decomposition}]  This follows from Propositions~\ref{prop:hyperplane}
and~\ref{prop:two-linear-spaces}
and Remark ~\ref{remark:prob-one}.
\end{proof}

Procedure~\ref{proc:numerical-irreducible-decomposition} returns ``Inconclusive" when there are two points in the returned witness sets that differ at every coordinate.  In this case, since we are working over $\mathbb{C}$, we cannot preclude ACR because we do not know how each component intersects the positive real orthant.  For example, 
a numerical irreducible decomposition of the system in Example~\ref{ex:generalized-shinar--Feinberg} returns three witness sets, one for each of the three hyperplanes $\{x_B=0 \}$, $\{x_A - \sqrt{\kappa_1/\kappa_2} = 0 \}$, $\{x_A + \sqrt{\kappa_1/\kappa_2} = 0 \}$. After removing the witness set that corresponds to the hyperplane $\{x_B=0 \}$, the remaining witness points (one for each remaining hyperplane) do not agree on any coordinate, however, the corresponding system does in fact have ACR in species A.

In practice, homotopy continuation software that computes numerical irreducible decompositions use random coefficients to define the intersecting linear spaces, thus, since it cannot be guaranteed that the linear spaces are indeed generic, they are ``probability one" algorithms.

\begin{example}(Lee {\em et al.} \cite{lee2003roles} \ Wnt model)
Here we use Procedure~\ref{proc:numerical-irreducible-decomposition} 
to test whether a model of Wnt pathway has ACR. The network is displayed in the input below. 
The model describes known interactions between core components of the canonical pathway where species can become active or inactive, denoted by a subscript $a$ or $i$, respectively.
 As far as we are aware, our analysis is the first testing of ACR in the Lee {\em et al.}\ Wnt model.  To do so, we use {\tt Macaulay2} \cite{M2} calling the {\tt ReactionNetworks.m2} package \cite{reactionnetworks}, the {\tt NumericalAlgebraicGeometry.m2} package \cite{leykin2011numerical}, and the following code.  In the code, we substitute the reaction rates with random rational positive values.  We use $\delta = 10^{-8}$ since that is the default error tolerance in {\tt NumericalAlgebraicGeometry.m2}.

{\footnotesize
\begin{verbatim}
RN=  reactionNetwork({"Di<-->Da", "Ya<-->Yi", "Da+Yi-->Da + G + CNA",  "G + CNA --> Yi", 
"Yi --> G + CNA", "A+N<-->CNA", "Ya+X<-->CXY", "CXY-->CYXp", "CYXp --> Xp + Ya",	"Xp-->0", 
"X<-->0", "0<-->N", "X+T<-->CXT", "X+A<-->CXA"}, NullSymbol => "0");
R = createRing(RN,QQ);
I = ideal( subRandomReactionRates RN);
S=QQ[RN.ConcentrationRates];
J = sub(I, S);
NV1 = numericalIrreducibleDecomposition(J);
Components1=components NV1;
\end{verbatim}
}

This returns to us the following numerical variety with a single witness set representing a single irreducible component of dimension 4 and degree 7:

{\footnotesize
\begin{verbatim}
i1: peek NV1
o1= NumericalVariety{4 => {(dim=4,deg=7)}}
\end{verbatim}
}

Taking a peek at the witness point of the component gives us

{\footnotesize

\begin{verbatim}
i1: (points Components1_0)_0
o1 =  {-.257725+.843632*ii, -1.23708+4.04943*ii, -.351557-.187174*ii, -2.87105-1.52859*ii, 
-.415284+3.13649*ii, -1.6121-.443075*ii, -.286596-.0787689*ii, 3, -.256615+.671232*ii, 
.479671-.417656*ii, 1.5989-1.39219*ii, 1.19918-1.04414*ii, -1.52228+3.79967*ii, 
-.493674-.456425*ii, 3.16043-4.30398*ii}
\end{verbatim}
}

We can see that there is a $3$ in the 8th entry, which corresponds to Axin ($N$). In fact, in this example, all seven witness points contains a $3$ in the 8th entry.  
A second witness point can be inspected using the following command, but we do not present the output here: 

{\footnotesize

\begin{verbatim} 
(points Components1_0)_1 \end{verbatim} 
}



Comparing witness points in this first decomposition suggests that we may have ACR.
Following Procedure~\ref{proc:numerical-irreducible-decomposition}, we can confirm ACR by performing a second numerical irreducible decomposition.  This amounts to choosing a new generic linear space $V_{\C}(L_{4,1})$.  Again, we get a numerical variety of dimension 4 and degree 7. 

{\footnotesize
\begin{verbatim}
i1: NV2 = numericalIrreducibleDecomposition(J);
i2: peek NV2
o2= NumericalVariety{4 => {(dim=4,deg=7)}}
\end{verbatim} 
}

Again, every witness point has a 3 in the 8th entry; one such point is shown below.

{\footnotesize
\begin{verbatim}
i1: (points Components2_0)_1
o1 =  {-.691939+.081125*ii, -3.32131+.389401*ii, .434359+.0253714*ii, 3.54727+.2072*ii, 
-.182904-.70573*ii, .307375-3.42684*ii, .054645-.609217*ii, 3, .201325-.0071503*ii, 
.194731+.0044491*ii, .649103+.014830*ii, .486827+.0111228*ii, -3.10059+6.38549*ii,
-.132244+.29891*ii, .16613-3.07604*ii}
\end{verbatim} 
}

Since each coordinate corresponding to Axin ($N$) is  within $\delta$ for all points in both witness sets, we conclude that the system has \emph{Numerical ACR in N within $\delta = 10^{-8}$.} 
\end{example}

\subsection{ACR via numerically computing points on each real connected component}\label{subsec:real-connected-components}


While finding the real connected components of a real algebraic set symbolically is challenging, we can gain information about the connected components by using numerical algebraic geometry.  In particular, using results from \cite{dufresne2019sampling, hauenstein2013numerically, seidenberg1954new}, we can find at least one point on each real connected component, and by repeating this procedure we can find multiple points on each component. In fact, we can find a provably dense set of points on each real connected component~\cite{dufresne2019sampling}. This yields a numerical algorithm for precluding ACR (Procedure \ref{proc:denseset}).

Given a reaction system $(G, \kappa)$ with ideal $I(G, \kappa)$ generated by $f_{\kappa}(x)= (f_{\kappa}(x)_1, f_{\kappa}(x)_2, \dots, f_{\kappa}(x)_n)$,  
we can find a real point on every component of $V_{\mathbb R} (I(G, \kappa))$ using an approach by Seidenberg~\cite{seidenberg1954new}. In that approach, we let $g = f_{\kappa}(x)_1^2 + f_{\kappa}(x)_2^2 + \dots + f_{\kappa}(x)_n^2$, and choose a random point $w \in \R^n$. We then solve the following system in indeterminates $x$ and $\lambda=(\lambda_0,\lambda_1)$:
\begin{align} \label{eq:realsystem1}
 g(x)~&=~0, \\
\lambda_0(x-w) + \lambda_1 \nabla g   ~&=~0~. \notag
\end{align}

For the system~\eqref{eq:realsystem1}, which is  referred to as the \emph{Fritz John optimality conditions}, the solutions will include the closest point on $V_{\mathbb R} (I(G, \kappa))$ to the point $w$.   As described in \cite{hauenstein2013numerically}, solutions to the Fritz John optimality conditions can be found through constructing and, in turn, tracking a homotopy.  For $\beta \in \mathbb C^{n-d}$, the homotopy  
we want to consider is 
\begin{equation} \label{eq:homotopy}
H_{w, \beta} (x, \lambda, t) = \left[\begin{array}{c}g(x) - t \beta \\ \lambda_0(x-w) + \lambda_1 \nabla g   =0 \end{array}\right].
\end{equation}

For generic choices of $w$ and $\beta$, the homotopy $H_{w, \beta}$ is a well-constructed homotopy~\cite[Proposition 3.2]{dufresne2019sampling}.  In practice, this means that we can find the solutions to the system in \eqref{eq:homotopy} by tracking a finite number of paths from the solutions $H_{w, \beta} (x, \lambda, 1)$ at $t=1$ to the target solutions $H_{w, \beta} (x, \lambda, 0)$ at $t=0$, 
which are solutions of~\eqref{eq:realsystem1}. 
Such homotopies can be tracked using polynomial homotopy continuation solvers {\tt Bertini} \cite{BHSW06}, {\tt PHCpack} \cite{verschelde1999algorithm}, and {\tt HomotopyContinuation.jl} \cite{breiding2018homotopycontinuation}. 


Repeatedly choosing pairs $(w, \beta)$, tracking the resulting homotopy~\eqref{eq:homotopy}, and projecting target solutions at $t=0$ onto the $x$-coordinates gives us a way to find multiple points on each component.  Moreover, using Algorithm 4.2 in \cite{dufresne2019sampling}, 
we can find a dense sample of points on $V_{\mathbb R} (I(G, \kappa))$ in a given region, 
as follows. 

\begin{lemma} \label{lem:dense-sampling}
There is an algorithm, which we call {\tt DenseSamplingAlgorithm}, with the following input and output:
\begin{itemize}
	\item {\sc Input}: A set of real polynomials $f=(f_1,f_2,\dots, f_k)$ on $n$ variables, a box $R = [a_1, b_1] \times \cdots \times [a_n, b_n]$, a sampling density $\epsilon>0$, and 
 an estimation error $0 \leq \delta < \epsilon$.
	\item {\sc Output}: A list $A$ of points that form an $(\epsilon, \delta)$-sample of $V_{\mathbb R} (f) \cap R$.
\end{itemize}
\end{lemma}

Here, an 
{\em $(\epsilon, \delta)$-sample} means (i) every point in $A$ is within $\delta$ of a point in $V_{\mathbb R} (f) \cap R$, and (ii) every point in $V_{\mathbb R} (f) \cap R$ is within $\epsilon$ of a point in $A$.

The next procedure shows how to numerically preclude ACR using {\tt DenseSamplingAlgorithm}.  This method can also be used to gain a better understanding of the components of $V_{\mathbb R} (\sqrt[>0]{I(G, \kappa)})$.  
\begin{proc}
   \label{proc:denseset} 
 Numerically precluding ACR by computing points on each real connected component 
 \end{proc}
\begin{itemize}
 \item[Input.] A reaction system $(G, \kappa^*)$ with ideal $I(G, \kappa^*)$ generated by steady-state equations $f_{\kappa^*}(x)~=~ (f_{\kappa^*}(x)_1, f_{\kappa^*}(x)_2, \dots, f_{\kappa^*}(x)_n)$, a box $R = [a_1, b_1] \times \cdots \times [a_n, b_n]$ with $0<a_i<b_i$, a sampling density $\epsilon$, and an estimation error $\delta$  with $0 \leq \delta < \epsilon$.

\item[Output.] ``No Numerical ACR" or ``Inconclusive"  \end{itemize}

\begin{itemize}

 \item[Step~1.] Let $Sols$ be the output of {\tt DenseSamplingAlgorithm}[$f_{\kappa}(x), R, \epsilon, \delta$].



 \item[Step~2.] If 
 $\lvert Sols \rvert=0$ or $1$, or if 
 there exists a coordinate $x_i$ such that for all the points in $Sols$ the absolute value of pairwise differences of the $i$th coordinates is less than $\delta$, then return ``Inconclusive".
 \item[Step~3.] Else, return ``No Numerical ACR".
\end{itemize}

\begin{proof}[Proof of correctness of Procedure~\ref{proc:denseset}]  The correctness of the procedure follows from \cite{dufresne2019sampling}, Corollary \ref{cor:ACR-real-components}, and the fact that the box $R$ is contained in the positive orthant.
\end{proof}

 Procedure~\ref{proc:denseset} computes
a sample of $V_{\mathbb R} (\sqrt[>0]{I(G, \kappa)})$ that is dense when intersected with a box. Yet, to preclude ACR, it is sometimes enough to sample only two sets of points.

\begin{example}\label{ex:sampling-real-points}
Here we consider the  one-site distributive phosphorylation cycle:

     
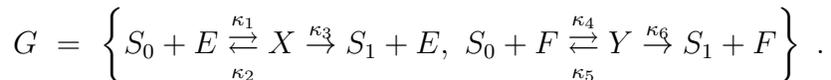
\[
G~=~\left\{ 
	S_{0}+ E \overset{\kappa_1}{\underset{\kappa_2} \rla } X \overset{\kappa_3} \rightarrow
	S_1 + E,~  S_{0}+ F\overset{\kappa_4}{\underset{\kappa_5} \rla } Y \overset{\kappa_6} \rightarrow
	S_1 + F 
 \right\}~.
\]
We let $x_1, x_2, x_3, x_4, x_5,$ and $x_6$ represent the concentrations of $S_0$, $E$, $X$, $S_1$, $F$, $Y$, respectively.

We pick two random points  
$w_1=\left( \frac{2}{3}, 6, \frac{2}{7}, \frac{1}{2}, 2, \frac{5}{2}\right) $ and $w_2=\left( \frac{3}{4}, \frac{1}{3}, \frac{7}{6}, \frac{5}{3}, \frac{3}{4}, \frac{5}{3}\right)$.  Setting up and tracking the homotopy in equation \eqref{eq:homotopy} twice returns the following two sets of real solutions: 
{\small
\begin{align*}
Sols_1 ~&=~
        \{ 
        (1.35308, 1.08525, .734218, 1.58447, .273065, .262221), 
        \\
        & \quad \quad \quad 
        (-.115201, -.190334, .0109633, .0497815, .129778, .00391547) \}
        \\
Sols_2 ~&=~ 
    \{(.390274, 5.98197, 1.1673, .348599, 1.97325, .416894), 
        \\
        & \quad  \quad \quad    
    (.617545, .00506007, .00156241, -.0576906, -.0159594, .000558004) \}~.
\end{align*}
}
Each set contains one positive solution.  The positive solutions differ in every coordinate, thus, we can numerically preclude ACR for the one-site distributive phosphorylation cycle. {\tt Macaulay2} code for this example can be found in Appendix \ref{sec:sampling-real-points}.
\end{example}


\section{Discussion} 
\label{sec:disc}

This paper gives an
overview into the algebra and geometry of absolute concentration robustness. 
 While at first glance, the property of ACR seems like it should be straightforward to check algorithmically, this paper  highlights the subtleties involved.  Indeed, the problem of detecting ACR ends up to be a nuanced problem in real algebraic geometry: \emph{Does every positive real component of a complex algebraic variety lie in a hyperplane of the form $\{ x_i - \alpha$\}? }

 The motivation for this study is biochemical reaction networks; however, there may be other applications that distill down to the same question.  Thus, this manuscript can serve as guide in those cases.  For example, there are several techniques that would be good first steps, while others would be good  second steps.  Other techniques would be valuable for in-depth case studies of  parameterized polynomial systems.

The easiest to implement approaches in the setting of biochemical reaction networks are those centered on graph theoretic and linear algebra tests.  For example, Shinar and Feinberg give a combinatorial condition on the graph when the reaction network has deficiency one, and they show that ACR is impossible when the network has deficiency zero \cite{shinar2010structural}.  A next approach would be those of \cite{beatriz-elisenda} and \cite{elisenda-oskar-beatriz}, discussed in Remarks \ref{rem:local-acar} and \ref{rem:hyperplane-check}, that give a necessary linear algebra condition for ACR, and thus, can be used to quickly preclude ACR.

 From the point-of-view of computational algebraic geometry, given a reaction system $(G, \kappa)$ with steady-state ideal $I$, the first step in checking for ACR would be through previously suggested methods as captured in Proposition \ref{prop:sufficient-conditions-ACR-using-steady-state-ideal} and Remark \ref{rem:gb}. In particular, we suggest to first compute a Gr\"{o}bner basis of $I$ and check for an element of the form $x_i - \alpha$.  In most examples, this was enough for us to detect ACR.  If that doesn't work, then we suggest computing the saturation $I: (x_1x_2 \cdots x_n)^ \infty$ and checking for an element of the form $x_i - \alpha$.  Then, if that doesn't work, compute each elimination ideal $I \cap \mathbb Q[x_i]$ and check whether there is a generator $g$ of any of the elimination ideals such that $g$ has a unique positive root.  While these three computations constitute reasonable first steps, as we saw in Example~\ref{ex:wrong2}, none of the conditions in Proposition \ref{prop:sufficient-conditions-ACR-using-steady-state-ideal} are necessary conditions for ACR.  Indeed, this point is the main motivation for this paper.
 
 If the methods suggested by Proposition \ref{prop:sufficient-conditions-ACR-using-steady-state-ideal} fail, either due to inconclusive results or computational limitations, the next steps would be those proposed in Sections \ref{sec:ideal_decomp}--\ref{sec:algorithm}.  If the system is small, then a next step could be Procedure \ref{prod:ideal-decomposition}, which relies on computing a decomposition of the positive-restriction ideal and successively  augmenting the ideals with Jacobian minors. However, it may be that the system is too large to effectively compute decompositions in a reasonable time.  In this case, or in the case that we have irrational rate constants $\kappa$ or we expect irrational ACR-value $\alpha$, then we suggest numerically checking for ACR through a numerical irreducible decomposition (Procedure \ref{proc:numerical-irreducible-decomposition}).  For those looking for a more complete algebraic or geometric description of a specific network in respect to ACR, then we suggest Algorithm \ref{Algorithm}, which examines the leading coefficients of  generators of elimination ideals obtained through Gr\"{o}bner bases, as well as Procedure~\ref{proc:numerical-irreducible-decomposition}, which returns a dense sample of each real component intersected with the positive orthant. 


 Finally, Section~\ref{sec:specialized-GB} raises some interesting questions related to computational algebraic geometry and ACR.  
 In particular, we asked whether zero-divisor ACR can always be detected from a Gr\"obner basis with respect to an  elimination order
 (Conjecture~\ref{conj:algorithm}).  
 An affirmative answer would certify that Algorithm~\ref{Algorithm} always succeeds in detecting zero-divisor ACR, thereby strengthening the connection between ACR and computational algebraic geometry.

\subsection*{Acknowledgements}
{\small
This project began at a SQuaRE (Structured Quartet Research Ensemble) at the American Instiute of Mathematics (AIM), and the authors thank AIM for providing financial support and an excellent working environment. 
The authors are grateful to Helen Byrne, 
Antonio Montes, 
Daniel Perrucci, 
Yue Ren, 
Marie-Françoise Roy, 
Frank Sottile, and Mike Stillman for helpful discussions.
The authors thank Alicia Dickenstein 
for inspiring some of the ideas that appear in Section~\ref{sec:comp_considerations} and Appendix~\ref{sec:cpx-acr-appendix}.
EG was supported by the NSF (DMS-1945584). 
Part of this research took place while EG was visiting the Institute for Mathematical and Statistical Innovation (IMSI), which is supported by the NSF (Grant No. DMS-1929348).
NM was partially supported by the Clare Boothe Luce Program from the Henry Luce Foundation. 
MPM was partially supported by UBACYT 20020220200166BA and CONICET PIP 11220200100182CO. 
AS was partially supported by the 
NSF (DMS-1752672) and the 
Simons Foundation (\#521874). HAH was supported by UK Centre for Topological Data Analysis EPSRC grant EP/R018472/1. and funding from the Royal Society RGF\textbackslash{}EA\textbackslash{}201074 and UF150238. For the purpose of Open Access, the authors have applied a CC BY public copyright licence to any Author Accepted Manuscript (AAM) version arising from this submission.  
The authors thank several reviewers whose detailed suggestions helped improve this work.
}

\appendix

\section{The Real Nullstellensatz for rational coefficients} \label{sec:appendix-radical}

The Real Nullstellensatz (Proposition~\ref{prop:real-nullstellensatz}) is well known, but we did not readily find the version we need in the literature.  Accordingly, the aim of this appendix is to show how Proposition~\ref{prop:real-nullstellensatz} follows from a result in~\cite{lombardi-perrucci-roy}.

The authors of~\cite{lombardi-perrucci-roy} considered an ordered field $\mathbf{K}$ and a real closed extension of $\mathbf{K}$, denoted by $\mathbf{R}$ (for instance, $\mathbf{K}=\mathbb{Q}$ and $\mathbf{R} = \mathbb R$). 
Next, they defined the following subset of $\mathbf{K}[x_1,x_2,\dots, x_n]$:
\[\mathcal{N}(\varnothing) ~:=~ 
\left\{
    \sum_{i=1}^m\omega_iV_i^2 ~|~ 
     m \in \mathbb{Z}_{\geq 0},~
    \omega_1, \omega_2,\dots, \omega_m
    \in\mathbf{K}_{>0},\, V_1,V_2,\dots, V_m \in \mathbf{K}[x_1, x_2,\dots, x_n] 
    \right\}~.
    \]
With this notation, the following result is called the Real Nullstellensatz in~\cite{lombardi-perrucci-roy}:

\begin{proposition}[Theorem~1.11 in~\cite{lombardi-perrucci-roy}]\label{thm:Real_Radical}
    Let $Q,P_1,P_2,\dots, P_s\subseteq \mathbf{K}[x_1,x_2,\dots, x_n]$, and let $J$ denote the ideal in $\mathbf{K}[x_1,x_2,\dots,x_n]$ generated by $P_1,  P_2, \dots, P_s$. If $Q$ vanishes on the common zero set of $P_1, P_2, \dots,P_s$ in $\mathbf{R}^n$, then there exist $e \in \mathbb{Z}_{>0}$, $N\in \mathcal{N}(\varnothing)$, and $Z \in J$ such that 
    \[Q^{2e}+N=Z~.\]
\end{proposition}

Now we use Proposition~\ref{thm:Real_Radical} (specifically, the case of   $\mathbf{K}=\mathbb{Q}$) to prove Proposition~\ref{prop:real-nullstellensatz}.

\begin{proof}[Proof of Proposition~\ref{prop:real-nullstellensatz}]

Let $J$ be an ideal of $\mathbb{Q}[x_1,x_2,\dots, x_n]$.  We must show that $	\sqrt[{\mathbb R}]{J}$ (the real radical of $J$) equals the vanishing ideal of $V_{\R}(J)$ (the real variety of $J$).  

For the first containment, let $P\in\sqrt[{\mathbb R}]{J}$ and let $x\in V_{\R}(J)$. We must show that $P(x)=0$.
As  $P\in\sqrt[{\mathbb R}]{J}$, there exist $m,\ell\in \mathbb{Z}_{\geq 0}$ and $h_1,h_2, \dots, h_\ell \in \Q[x_1, x_2, \dots, x_n]$ such that  $P^{2m}+\sum_{i=1}^\ell h_i^2 \in J$.  Thus, as $x\in V_{\R}(J)$, we have the following sum of squares:
 \[ (P(x)) ^{2m} + \sum_{i=1}^\ell (h_i(x))^2~=~0~.\]
We conclude, as desired, that $P(x)=0$.

    For the reverse containment, consider the ordered field $\mathbf{K}=\Q$. Every positive rational number is a sum of squares, as follows:
    \[ \frac{p}{q}
    ~=~p\cdot q\cdot \frac{1}{q^2}
    ~=~
    \underset{p\cdot q \text{ times}}{\underbrace{\frac{1}{q^2}+\dots +\frac{1}{q^2}}}~,
    \]
    for $p,q\in\Z_{> 0}$. 
    This fact readily implies that $\mathcal{N}(\varnothing)$ consists of sums of squares, as follows: 
    $\mathcal{N}(\varnothing)=\{\sum_{i=1}^m V_i^2 ~|~  V_i\in \Q[x_1,x_2,\dots, x_n] \}$. 
    This property of $\mathcal{N}(\varnothing)$, together with Proposition~\ref{thm:Real_Radical} (applied to  $\mathbf{K}=\mathbb{Q}$ and $\mathbf{R} = \mathbb R$), imply the following: for every $P$ in the vanishing ideal of $V_{\R}(J)$, 
    there exist
    $e,m \in \mathbb{Z}_{> 0}$ and $V_1,V_2,\dots, V_m \in \Q[x_1, x_2, \dots, x_n]$
     such that 
    the sum $P^{2e}+\sum_{i=1}^m V_i^2$ is in $J$,  that is, $P\in\sqrt[{\mathbb R}]{J}$.  Hence, the ideal of $V_{\R}(J)$ is contained in $\sqrt[{\mathbb R}]{J}$.  
\end{proof}

\section{Complex-number ACR} \label{sec:cpx-acr-appendix}

In this appendix, we widen our scope to the (algebraically closed field of the) complex numbers $\C$, where in general we lose chemical significance, but gain much from the theoretical and computational point of view.  Specifically, we introduce the notion of \emph{complex-number absolute concentration robustness (CACR)}, 
which -- in contrast to ACR -- can be readily detected from the steady-state ideal (Proposition~\ref{prop:sat} and Algorithm~\ref{algo:CACR}).  

In short, CACR 
generalizes the notion of ACR, by 
broadening the consideration from all positive zeros to all complex zeros with nonzero coordinates. 

\begin{definition}[CACR for mass-action systems] \label{def:CACR}
A mass-action system $( G, \kappa)$ with $n$ species and mass-action ODEs as in~\eqref{eq:ODE-mass-action}, has:
	\begin{enumerate}
	\item  \emph{complex-number absolute concentration robustness (CACR) in species $X_i$} if, for every $x^* \in (\C ^*)^n$ satisfying 
$f_{\kappa}(x^*)_1=f_{\kappa}(x^*)_2=\dots=f_{\kappa}(x^*)_n=0$, the value of $x_i$ is the same. 
This value of $x_i$ is the {\em CACR-value}.  
	\item {\em vacuous CACR} if there does {\bf not} exist $x^* \in (\C ^*)^n$ satisfying 
$f_{\kappa}(x^*)_1=f_{\kappa}(x^*)_2=\dots=f_{\kappa}(x^*)_n=0$.
	\end{enumerate}
\end{definition} 

As we do for ACR, we say that $(G,\kappa)$ has CACR if it has CACR in some species.
It is immediate from the relevant definitions that CACR implies (possibly vacuous) ACR. 

\begin{proposition}[CACR implies ACR] \label{prop:cacr-vs-acr}
If a mass-action system $( G, \kappa)$ has CACR in some species $X_i$, then the system also has ACR in $X_i$.  
	Moreover:
	\begin{enumerate}
		\item If  $( G, \kappa)$ has non-vacuous CACR in $X_i$ with CACR-value a positive real number, then $( G, \kappa)$ has (possibly vacuous) ACR in $X_i$.
		\item  If  $( G, \kappa)$ has vacuous CACR or has non-vacuous CACR in $X_i$ with CACR-value that is {\bf not} a positive real number, then  $( G, \kappa)$ has vacuous ACR.
	\end{enumerate}
\end{proposition}

\begin{example}
Consider the following network:
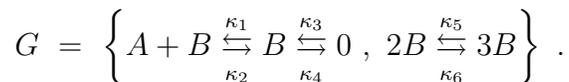
\[
G~=~\left\{ A+B  \overset{\kappa_1}{\underset{\kappa_2}\lra}
B \overset{\kappa_3}{\underset{\kappa_4}\lra} 0~, ~ 
2B \overset{\kappa_5}{\underset{\kappa_6}\lra} 3B
\right\}~.
\]
It is straightforward to check that $G$ has no conservation laws and that for any positive steady state $(x_A^*,x_B^*)$, we have $x_A^*=\kappa_2/\kappa_1$.  Thus, $G$ has ACR in species $A$.  
Also, $x_B^*$ can take up to three positive-steady-state values.  For instance, if $\kappa_3=11$, $\kappa_4=\kappa_5=6$, and $\kappa_6=1$, then $x_B^*$ takes three values (namely, $x_B^*=1$, 2, and 3), yielding three positive steady states $(x_A^*=\kappa_2/\kappa_1,~x_B^*)$. Hence, $G$ is multistationary and does not have ACR in $B$.
It is straightforward to check that these systems also have CACR in species $A$ (with CACR-value $x_A^*=\kappa_2/\kappa_1$).
\end{example} 


\begin{example}[Example~\ref{ex:singular}, continued] \label{ex:singular-2}
 For network $G=\{  3A \overset{1}{\to} 4A,~ A+2B \overset{1}{\to} 2A+2B,~ 
	2A \overset{2}{\to} A,~ A+B \overset{4}{\to} B,~ 
	A \overset{5}{\to} 2A,~ 2A+B \overset{1}{\to} 2A+2B,~ 3B \overset{1}{\to} 4B,~ 
	A+B \overset{2}{\to} A,~ 2B \overset{4}{\to} B,~ 
	B \overset{5}{\to} 2B\}$ we saw that the only positive root of $f_A=f_B=0$ is $(x_A, x_B)=(1,2)$, and so this system has ACR in both species. However, $(4,2+3i)$, $(4,2-3i)$, $(-2,2+3i)\in(\C^*)^2$ are also roots of $f_A=f_B=0$ and so the system does not have CACR. 
\end{example}

Another example with ACR but not CACR is Example~\ref{ex:ideals-network}.

What we gain from expanding the definition of ACR to CACR is \emph{necessary} and sufficient conditions for CACR. 
Moreover, in theory, CACR can be detected algorithmically (if computations over $\mathbb C$ are possible).  
We record this claim in Proposition~\ref{prop:sat} and Algorithm~\ref{algo:CACR} below, which make use of radical ideals~\eqref{eq:radical} and saturation ideals (Definition~\ref{def:colon}).
We also extend steady-state ideals to reside over $\mathbb{C}$ as follows:

\begin{definition} \label{def:complex-steady-state-ideal}
For a mass-action system $(G, \kappa)$, the {\em complex-steady-state ideal} is the ideal of $\mathbb C [x_1, x_2, \ldots , x_n] $ 
	generated by the right-hand sides of the ODEs in~\eqref{eq:ODE-mass-action}. 
\end{definition}



For part (3) of the next result, recall that any ideal of $\mathbb{C}[x_i]$ (with only one variable $x_i$)
is principally generated, so we may speak of ``a generator''.  Also, the proof makes use of Hilbert's Nullstellensatz (e.g., see~\cite{CLO15}).

\begin{proposition}[CACR and ideals] \label{prop:sat}
Let $( G, \kappa^*)$ be a mass-action system with $n$ species and complex-steady-state ideal $I$. 
Let $\mathfrak m:=  x_1 x_2 \cdots x_n$, and let $ 1 \leq i \leq n$.
Then $( G, \kappa^*)$ has vacuous CACR if and only if $ \sqrt{(I : \mathfrak m^\infty)} = \langle 1 \rangle $.  On the other hand, for $\alpha \in \mathbb C^*$, the following are equivalent:
	\begin{enumerate}
	\item $( G, \kappa^*)$ has non-vacuous CACR in species $X_i$ with CACR-value $\alpha$,
	\item $x_i - \alpha $ is in $ \sqrt{(I : \mathfrak m^\infty)}$, and 
	$ \sqrt{(I : \mathfrak m^\infty)} \neq \langle 1 \rangle$,  
	\item $x_i - \alpha$ is a generator of $ \sqrt{(I : \mathfrak m^\infty) \cap \mathbb{C}[x_i] }$. 
	\end{enumerate}
\end{proposition}

\begin{proof} 
The equivalence of 
vacuous CACR and $ \sqrt{(I : \mathfrak m^\infty)} = \langle 1 \rangle $ is readily deduced from \cite[Theorem~10 of \S~4]{CLO15} and the Weak Nullstellensatz.  
It is also straightforward to see that (2) and (3) are equivalent, so we show the equivalence of (1) and (2) below.

First assume that the system has non-vacuous CACR in species $X_i$ with CACR-value $\alpha$. 
Then, $\mathfrak m \, (x_i - \alpha)$ is identically zero over all the zeros of $I$ in $\C^n$. We deduce from Hilbert's Nullstellensatz that there exists a natural number $k$ such that $ \mathfrak m^k \, (x_i - \alpha)^k \in I$. Thus, $(x_i - \alpha)^k \in (I : \mathfrak m^\infty)$
and so 
$x_i - \alpha \in \sqrt{(I : \mathfrak m^\infty)}$.

Now assume that $x_i - \alpha$ is in $\sqrt{(I : \mathfrak m^\infty)}$. Then $\mathfrak m^k(x_i - \alpha)^{k} \in I$ for some $k$. Consider $\tilde{x}^*  \in (\C^*)^n$ for which $f_{\kappa}(\tilde{x}^*)_1=f_{\kappa}(\tilde{x}^*)_2=\dots=f_{\kappa}(\tilde{x}^*)_n=0$.  It follows that 
$(\tilde{x}^*_1 \tilde{x}^*_2 \dots \tilde{x}^*_n)^k(\tilde{x}^*_i- \alpha)^{k}=0$.  As $(\tilde{x}^*_1 \tilde{x}^*_2  \dots \tilde{x}^*_n) \neq 0$, we have $(\tilde{x}^*_i- \alpha)^{k}=0$, and so $\tilde{x}^*_i= \alpha$.  
Thus, we have CACR in $X_i$ with CACR-value $\alpha$.
\end{proof}

Proposition~\ref{prop:sat} yields Algorithm~\ref{algo:CACR} for detecting CACR.  The algorithm does not compute the radical ideal until Step~3 because it is cheaper to do computations in one variable. For details about the computations of these ideals and generators, see the book \cite{CLO15}.

\begin{algorithm}

\medskip

\noindent{\bf Input:} The complex-steady-state ideal $I$ of a mass-action system $(G, \kappa^*)$. 

\noindent{\bf Output:} The CACR-value if $(G, \kappa^*)$ has non-vacuous CACR in $X_i$; ``Vacuous ACR'' or  ``No CACR''  if $(G,\kappa^*)$ has, respectively, vacuous ACR or no CACR in $X_i$.

\medskip

\noindent Step~0: Let $\mathfrak m := x_1 x_2 \dots x_n$, where $n$ is the number of species of $G$.

\noindent Step~1: Compute the saturation ideal $(I: \mathfrak m^\infty)$.

\noindent Step~2: Compute the elimination ideal $I_i :=(I :  \mathfrak m^\infty) 
\cap \mathbb{C}[x_i]$.  

\noindent Step~3: Compute $\sqrt{I_i}$ and find a generator $g$.

\noindent Step~4:
	If $g$ has degree 1, return its root. 
	If $g$ has degree 0, return ``Vacuous CACR''.  
	Otherwise, return ``No CACR''. 


\caption{Algorithm for CACR.}\label{algo:CACR}
\end{algorithm}

\begin{proof}[Proof of correctness of Algorithm~\ref{algo:CACR}]
The correctness of this algorithm follows directly from Proposition~\ref{prop:sat} and the fact that the generator $g$ in Step~3 can not be the degree-1 polynomial $g=x_i$, due to having performed saturation in Step~1.
\end{proof}

\section{Numerically sampling real points} \label{sec:sampling-real-points}

Here we give Macaulay2 code for Example \ref{ex:sampling-real-points}.  We use the {\tt ReactionNetworks.m2} package where the one-site distributive phosphorylation network can be called as follows:

\begin{verbatim}
N = oneSiteModificationA();
\end{verbatim}

We can also use the {\tt ReactionNetworks.m2} package to fix random reaction rates and conserved quantities.

\begin{verbatim}
R = createRing(N,QQ);
I = trim ideal (subRandomReactionRates N);
S=QQ[N.ConcentrationRates];
J = sub(I, S);
\end{verbatim}

We then can then choose a random point and use {\tt NumericalAlgebraicGeometry.m2} to solve the start system of the homotopy described in equation \ref{eq:homotopy}.

\begin{verbatim}
w = apply(6, i-> random QQ);

gen=flatten entries gens J;     
var=flatten entries vars ring J;
N=length var;     
D=length gen;
T=QQ[var| toList (lambda_1..lambda_(D))];
K=sub(J,T); 
var=apply(N,i->sub(var_i,T));
gen=apply(D, i->sub(gen_i, T));
Jc = jacobian (K);

beta = apply(D, i-> random QQ);

Hstart = 
          apply(#gen, i-> gen_i+beta_i) |
          apply(N, i-> ((var_(i)-w_(i)) + sum (apply (D, j -> lambda_(j+1)*Jc_(i,j) ) ) ) ); 
	  
sols = solveSystem (Hstart);
\end{verbatim}

Finally, we can construct the target system and track the start solutions to our target solutions.  We finish by applying a real filter to the target solutions.

\begin{verbatim}
Htarget =  apply(#gen, i-> gen_i) |
          apply(N, i-> ((var_(i)-w_(i)) + sum (apply (D, j -> lambda_(j+1)*Jc_(i,j) ) ) ) ) 	  

solsT = track(Hstart, Htarget, sols, gamma=>0.6+0.8*ii)  

realPoints(solsT)
\end{verbatim}

\end{document}